\documentclass[arxiv]{default}

\title{The average number of subgroups of elliptic curves over finite fields}
\def\titleHead{Average number of subgroups of elliptic curves over finite fields}
\author{Corentin Perret-Gentil}


\address{Centre de Recherches Mathématiques, Université de Montréal, Canada}
\curraddr{Zürich, Switzerland}
\email{corentin.perretgentil@gmail.com}

\usepackage{scrtime}
\date{December 2018. Revised September 2019.}

\usepackage{longtable}

\newcommand{\Ell}{\mathcal{E}ll}

\newcommand{\eqpageref}[1]{\eqref{#1} p. \pageref{#1}}

\newcommand{\oO}[1]{\left(1+O \left(#1\right)\right)}

\begin{document}

\begin{abstract}
  By adapting the technique of David, Koukoulopoulos and Smith for computing sums of Euler products, and using their interpretation of results of Schoof \emph{à la} Gekeler, we determine the average number of subgroups (or cyclic subgroups) of an elliptic curve over a fixed finite field of prime size. This is in line with previous works computing the average number of (cyclic) subgroups of finite abelian groups of rank at most $2$. A required input is a good estimate for the divisor function in both short interval and arithmetic progressions, that we obtain by combining ideas of Ivić--Zhai and Blomer. With the same tools, an asymptotic for the average of the number of divisors of the number of rational points could also be given.
\end{abstract}

\maketitle

\tableofcontents

\section{Introduction}

\subsection{Counting subgroups}

Given a nontrivial finite group $G$, we let $s(G)$ and $c(G)$ be its number of subgroups, resp. cyclic subgroups.

We have the trivial bounds $2\le s(G) \le 2^{|G|}$ and $1\le  c(G) \le |G|$, while Borovik--Pyber--Shalev \cite[Corollary 1.6]{BPS96} have shown the general upper bound
\begin{equation}
  \label{eq:BPS}
  s(G)\le\exp \left(\frac{(\log{|G|})^2}{\log{2}}\left(\frac{1}{4}+o(1)\right)\right).
\end{equation}

For certain groups, such as finite abelian groups, explicit formulas for $s(G)$ and $c(G)$ are well-known (see e.g. \cite{Steh92,Tar10} or \eqref{eq:sExplicit} and Propositions \ref{prop:scRank2}, \ref{prop:cRank2} below).

\subsection{Counting subgroups on average}
Given a finite family of finite groups $\Gc$, it may be interesting to understand the average
\begin{equation}
  \label{eq:Gcaverage}
  \frac{1}{|\Gc|}\sum_{G\in\Gc} h(G)
\end{equation}
when $h\in\{s,c\}$, i.e. the average number of (cyclic) subgroups.

\subsubsection{Finite abelian groups of rank at most $2$}
When $\Gc$ is the set of finite abelian groups of rank at most $2$ and size at most $x$, Bhowmik and Menzer \cite{BhowMenz97} determined that the average \eqref{eq:Gcaverage} for $h=s$ is given by
\begin{equation}
  \label{eq:averagerank2}
  \frac{A_1x(\log{x})^2+A_2x(\log{x})+A_3x+O(x^{31/43+\varepsilon})}{\sum_{r\le x}\tau(r)}=A_1\log{x}+A_2+o(1)
\end{equation}
as $x\to+\infty$, where $A_1,A_2,A_3$ are effective constants and $\tau$ is the number of divisors function. The error term was later improved by other authors (see e.g. \cite{Ivic97}).
\subsection{Elliptic curves}

Herein, we want to study the family $\Gc=\Ell(p)$ of (rational) isomorphism classes of elliptic curves defined over a finite field $\F_p$.

Since such curves are usually weighted by their number of automorphisms, we define a weighted version of \eqref{eq:Gcaverage} by
\begin{equation}\label{eq:GcaverageWeighted}
  h(\Ell(p)):=\frac{1}{p}\sum_{E\in\Ell(p)} \frac{h \left(E(\F_p)\right)}{|\Aut(E)|},
\end{equation}
where $h\in\{s,c\}$. We recall that $\Aut(E)\in\{2,4,6\}$ if $p\ge 5$, and $\sum_{E\in\Ell(p)}|\Aut(E)|^{-1}=p$.

Our main result is the following:
\begin{theorem}\label{thm:scEll}
  For any $A>0$, the weighted average number of subgroups of an elliptic curve over $\F_p$ is
<  \begin{eqnarray*}
    s(\Ell(p))&=&\left(\prod_{\ell\mid p-1}\left(1-\frac{1}{\ell(\ell^2-1)}\right)\right)\sum_{u\mid d_1\mid p-1} \frac{\varphi(u)\tau(d_1/u)}{d_1^3}\\
              &&\sum_{k\mid d_1^2/u}\left(\log \left(\frac{p+1}{uk^2}\right)+2\gamma\right)\frac{\varphi(k)+\delta_{k=1}}{k}\\
              &&\hspace{1cm}\prod_{\ell\mid k}\ell^{v_\ell(d_1)}\oO{\frac{1}{\ell}}\prod_{\substack{\ell\mid d_1}}\left(1-\frac{1}{\ell(\ell^2-1)}\right)^{-1}\\
              &&+ O_A \left(\frac{1}{(\log{p})^A}\right),
  \end{eqnarray*}
  with $\gamma=0.5772\dots$ the Euler--Mascheroni constant. The same result holds for $c(\Ell(p))$, after replacing $\varphi(u)$ by $(\varphi*\mu)(u)$, where $\mu$ is the Möbius function.
\end{theorem}
\begin{remark}
  The second Euler product, over $\ell\mid k$, can be given explicitly, without the error terms (cf. Proposition \ref{prop:localDensities} later on); the local factor at $\ell$ is a weighted sum of matrix densities in characteristic $\ell$. We also note that the first product $\prod_{\ell\mid p-1}(1-(\ell(\ell^2-1))^{-1})$ is the asymptotic probability that $E(\F_p)$ is cyclic, by Vl\u{a}du\c{t} \cite{Vla99} (see also \cite[Theorem 1.9]{DKS17})
\end{remark}

\begin{remark}\label{rem:primePower}
  The case of finite fields of higher degrees could be treated by extending the results of David--Koukoulopoulos--Smith (Proposition \ref{prop:PEH} below), as in the generalizations by Achter--Gordon \cite{AchGord17} or Kaplan--Petrow \cite{KaPe17} of the work of Gekeler \cite{Gek03}.
\end{remark}
\subsubsection{Order of magnitude}\label{subsec:orderMag}
If $E\in\Ell(p)$, then $E(\F_p)$ is a finite abelian group of rank at most $2$, i.e. there exist $d_1,d_2\ge 1$ such that
\[E(\F_p)\cong\Z/d_1\times\Z/d_1d_2.\]
Moreover, by the Hasse--Weil bound, $p_-\le d_1^2d_2\le p_+$. Therefore, one may expect from \eqref{eq:averagerank2} that $s(\Ell(p))$ is of order of magnitude $\log{p}$.

Indeed, averaging Theorem \ref{thm:scEll} over $p$, we find:
 \begin{proposition}\label{prop:averagep} For $h\in\{s,c\}$, we have
   \begin{eqnarray*}
     \frac{1}{\pi(x)}\sum_{p\le x} h(\Ell(p))&=&(C_h+o(1))\log(x+1)
   \end{eqnarray*}
   as $x\to\infty$, for some constant $C_h\ge 1$.
 \end{proposition}
 
 \begin{figure}
   \centering
   \includegraphics[scale=0.4]{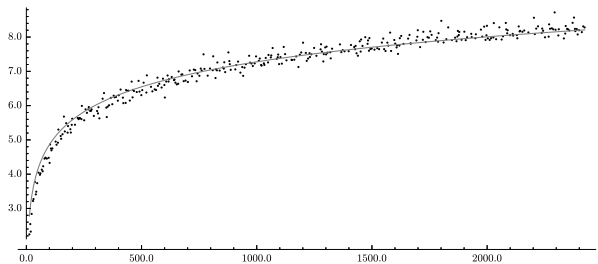}
   \caption{Graph of the average $y=\frac{1}{\pi(x)}\sum_{p\le x} s(\Ell(p))$ for $x\le 2423$ (black points) and of the least-squares fit $y=1.053\log{x}$ (gray line).}
 \end{figure}

 Pointwise, we find the following upper and lower bounds:
 \begin{proposition}\label{prop:upperlowerBounds} For $h\in\{s,c\}$ and $\varepsilon>0$,
   \begin{eqnarray*}
     h(\Ell(p))&\ll_\varepsilon&(\log{p})^{1+e^\gamma+\varepsilon}(\log_2{p})\sum_{d_1\mid p-1}\frac{\tau(d_1^2)}{d_1}\\
               &\ll&(\log{p})^{1+e^\gamma+\varepsilon}(\log_2{p})\min \left((\log{p})^4, \tau((p-1)^2) \frac{\sigma(p-1)}{p-1}\right),\\
     s(\Ell(p))&\gg_\varepsilon&\sum_{d_1\mid p-1} \frac{\sigma(d_1)}{d_1^2}\ge \frac{\sigma(p-1)}{p-1},\\
     c(\Ell(p))&\gg_\varepsilon&\sum_{d_1\mid p-1} \frac{(\sigma*\id)(d_1)}{d_1^2}\gg \frac{\sigma(p-1)}{p-1}.
   \end{eqnarray*}
 \end{proposition}
 \begin{remark}
   This is to be compared with the upper bound
   \[s(\Ell(p))\ll\exp \left(\frac{(\log{p})^2}{\log{2}}\left(\frac{1}{4}+o(1)\right)\right)\]
   that follows from the Hasse--Weil bound and the general result \eqref{eq:BPS} of Borovik--Pyber--Shalev, and to the bounds
   \[\tau(d_1d_2)\sigma(d_1)\ll s(\Z/d_1\times \Z/d_1d_2)\ll \tau(d_1^2d_2)\sigma(d_1) \qquad(d_1,d_2\ge 1)\]
   that are readily obtained from \eqref{eq:sExplicit} below.
 \end{remark}
 
\subsubsection{Number of divisors of $|E(\F_p)|$}
Using the same tools and ideas, one could also give an explicit formula (as in Theorem \ref{thm:scEll}) for
\begin{equation}
  \label{eq:tauE}
  \frac{1}{p}\sum_{E\in\Ell(p)} \frac{\tau(|E(\F_p)|)}{|\Aut(E)|},
\end{equation}
i.e. the average number of divisors of the number of rational points, and an average of this quantity over primes $p\le x$ (as in Proposition \ref{prop:averagep}), of the order of $x\log{x}$ as $x\to\infty$. Since this is easier than Theorem \ref{thm:scEll} and that we plan to go back to averages of the type \eqref{eq:tauE} in future work, we do not give further details here.

\subsection{Ideas and organization of the paper}

The idea of the proof of Theorem \ref{thm:scEll} is roughly the following:
\begin{enumerate}
\item Condition the sum \eqref{eq:GcaverageWeighted} on the $\F_p$-rational group structure, and use an explicit expression for the number of elliptic curves over $\F_p$ having a fix such structure, obtained by David--Kouloulopoulos--Smith \cite{DKS17}. Their work translates a result of Schoof \cite{Schoo87} into a product of local densities over the primes, as in the very insightful work of Gekeler \cite{Gek03};
\item Adapt ideas from \cite{DKS17} to compute the weighted sums of Euler products that arise, under the condition that $h$ is well-distributed in some arithmetic progression in short intervals. Under additional assumptions, the main term can itself be given as a sum of Euler products with explicit local factors;
\item Show that these conditions hold for $h=\{s,c\}$, using a known explicit expression for $h$ as a convolution of arithmetic functions including Euler's totient, the divisor function, and the Möbius function. For example, we have (see Proposition \ref{prop:scRank2} below):
  \begin{equation}
    \label{eq:sExplicit}
    s(\Z/d_1\times\Z/d_1d_2)=\sum_{u\mid d_1}\varphi(u)\tau(d_1/u)\tau(d_1^2d_2/u).
  \end{equation}
\end{enumerate}
To control the asymptotic errors, we will prove the following result on the mean square error of the approximation of the divisor function in short intervals and arithmetic progression (see Section \ref{sec:DeltaqAB} for the notations):
{
\renewcommand{\thetheoscounter}{\ref{thm:tauAPSI}}
\begin{theorem}
  Let $1\le A<B$, $1\le q\le \sqrt{A}$, and $\varepsilon>0$. We have
  \begin{eqnarray}
    &&\frac{1}{q}\sum_{a\in\Z/q}|\Delta(A,B,a,q)|^2\label{eq:divAPSI}\\
    &&\hspace{2cm}\ll_\varepsilon (qB)^\varepsilon
          \begin{cases}
            \frac{(B-A)^{1/2}}{q}\left(\frac{B^3}{A}\right)^{1/4}&:B-A\le \sqrt{B},\\
            \frac{(B-A)^{4/3}}{q^{4/3}}\left(\frac{B}{A}\right)^{1/3}&:\sqrt{B}\le B-A\le\sqrt{AB}.
          \end{cases}\nonumber                                                               
        \end{eqnarray}
      \end{theorem}     
\addtocounter{theoscounter}{-1}
    }
      This is obtained by combining ideas from Blomer \cite{Blo07} (for the arithmetic progression aspect) and Ivić--Zhai \cite{IvZh14} (for the short interval aspect)

\subsubsection{Structure of the paper}

In Section \ref{sec:strategy}, we give the general strategy to compute averages of the form \eqref{eq:GcaverageWeighted}, and obtain a result (Theorem \ref{thm:general}) for a certain class of functions $h$. The proofs are then given in Section \ref{sec:proofs}.

The remaining of the text focuses on the particular cases $h=s,c$. Section \ref{sec:DeltaqAB} is dedicated to proving the estimate \eqref{eq:divAPSI}. In Section \ref{sec:hs}, we deduce Theorem \ref{thm:scEll} for $h=s$ from Theorem \ref{thm:general}, and the same is done for $h=c$ in Section \ref{sec:hc}.

\subsubsection{Notations}
Throughout, we will employ the usual convention that $\varepsilon$ denotes a positive number as small as desired, whose exact size may change from one line to the next. Unless otherwise mentioned, the implied constants may depend on $\varepsilon$. For an integer $n$, we let $\rad(n)$ be the product of its prime factors without multiplicity, $P^+(n)$ its largest prime factor, $\omega(n)$ its number of distinct prime factors, and $\tau(n)$ (resp. $\sigma(n)$) the number (resp. sum) of its divisors. For $m\ge 1$, the notation $\log_m$ stands for the $m$th iterated natural logarithm.

\subsubsection*{Acknowledgements}

The author thanks Dimitris Koukoulopoulos, as well as his other colleagues in Montréal, for helpful discussions during this project. We also thank an anonymous referee for providing comments that improved the exposition. The author was partially supported by Koukoulopoulos' Discovery Grant 435272-2013 of the Natural Sciences and Engineering Research Council of Canada, and by Radziwill's NSERC DG grant and the CRC program.

\section{General strategy}\label{sec:strategy}

In this section, we consider a real-valued function $h$ on isomorphism classes of finite abelian groups of rank at most $2$, and we aim to compute the weighted average $h(\Ell(q))$ defined in \eqref{eq:GcaverageWeighted}. We recall that the proofs of the following results will be given in Section \ref{sec:proofs}.

For integers $d_1,d_2\ge 1$, we use the abbreviation\label{page:hd1d2}
\[h(d_1,d_2):=h(\Z/d_1\times\Z/d_1d_2)\]
and for any group $H$, we define
  \begin{equation}
    \label{eq:PEH}
    \P \left(E(\F_p)\cong H\right):=\frac{1}{p}\sum_{E\in\Ell(p)}\frac{\delta_{E(\F_p)\cong H}}{|\Aut(E)|}.
  \end{equation} 
\subsection{Conditioning on the group structure}

We start with an alternative expression for $h(\Ell(p))$.
\begin{proposition}\label{prop:fellqExplicit} We have
  \[h(\Ell(p))=\sum_{\substack{d_2\ge 1\\d_1\mid p-1\\p_-\le d_1^2d_2\le p_+}} h(d_1,d_2)\P \left(E(\F_p)\cong\Z/d_1\times\Z/d_1d_2\right).\]
  where $p_\pm:=(\sqrt{p}\pm 1)^2$.\label{page:qpm}
\end{proposition}

\subsection{Probability of having a given group of rational points}

Expressing a result of Schoof \cite[Lemma 4.8, Theorem 4.9]{Schoo87} using ideas of Gekeler \cite{Gek03}, David, Koukoulopoulos and Smith gave the probability \eqref{eq:PEH} as:
\begin{proposition}[{\cite[Theorem 1.7]{DKS17}}]\label{prop:PEH}
For $d_1,d_2\ge 1$, we have
\[\P \left(E(\F_p)\cong\Z/d_1\times\Z/d_1d_2\right)=f_\infty(p+1-d_1^2d_2,p)\prod_\ell f_\ell(d_1,d_2,p),\]
where $\ell$ runs over primes and
  \begin{eqnarray}
  f_\ell(d_1,d_2,p)&:=&\lim_{r\to+\infty} \frac{\left|\left\{g\in M_2(\Z/\ell^r) : \substack{\det(g)=p\\\tr(g)=p+1-d_1^2d_2\\ g\equiv 1\pmod{\ell^{v_\ell(d_1)}}\\ g\not\equiv 1\pmod{\ell^{v_\ell(d_1)+1}}}\right\}\right|}{\ell^{2r}(1-1/\ell^2)},\label{eq:fell}\\
    f_\infty(t,p)&:=&\frac{1}{p\pi}\sqrt{4p-t^2}\delta_{|t|<2\sqrt{p}}.\label{eq:finfty}
\end{eqnarray}
\end{proposition}
We record the following important information about the asymptotic matrix densities $f_\ell$:
\begin{proposition}[{\cite[Theorem 3.2]{DKS17}}]\label{prop:PEH2}~
  \begin{enumerate}
  \item The limit as $r\to\infty$ defining $f_\ell(d_1,d_2,p)$ stabilizes when $r>v_\ell(D_{d_1^2d_2,p})$, where we let $D_{a,p}:=(p+1-a)^2-4p$ for $a\in\Z$. \inlinelabel{eq:Dap}
\item If $\ell\nmid D_{d_1^2d_2,p}/d_1^2$, then
  \[f_\ell(d_1,d_2,p)=\frac{1}{\ell^{v_\ell(d_1)}}\left(1-\frac{1}{\ell^2}\right)^{-1}\left(1+\frac{\chi_{d_1,d_2,p}(\ell)}{\ell}\right)\]
  for the quadratic character $\chi_{d_1,d_2,p}=\legendre{D_{d_1^2d_2,p}/d_1^2}{\cdot}$.
\item If $d_1\mid p-1$,
  \[f_\ell(d_1,d_2,p)=\frac{1}{\ell^{v_\ell(d_1)}}\oO{\frac{1}{\ell}}.\]
  More precisely, $1\le f_\ell(d_1,d_2,p)\ell^{v_\ell(d_1)}\le 1+\frac{2}{\ell}\left(1+\frac{1}{\ell-1}\right).$ \inlinelabel{eq:fellPrecise}
\item\label{item:PEH2p} For $p$ large enough, we have
  \[f_p(d_1,d_2,p)=1+\frac{\delta_{p\nmid d_1^2d_2-1}}{p-1}.\]
\end{enumerate}  
\end{proposition}
\begin{remark}
  If $E(\F_p)\cong\Z/d_1\times\Z/d_1d_2$ and $p\mid d_1^2d_2-1$, then $E$ is supersingular. There are only $O(\sqrt{p}\log{p})$ isomorphism classes of such curves over $\F_p$ (see e.g. \cite[Theorem 14.18]{Cox89}); in particular, those can be ignored in \eqref{eq:GcaverageWeighted} up to introducing an error of size $O(p^{-1/2}\log{p})$.
\end{remark}
\subsection{Sum of Euler products}
By Propositions \ref{prop:fellqExplicit} and \ref{prop:PEH}, we have
\begin{eqnarray}
  h(\Ell(p))&=&\sum_{\substack{d_1,d_2\\p_-\le d_1^2d_2\le p_+}} w_{h,p}(d_1,d_2)\prod_{\ell} f_\ell(d_1,d_2,p),
  \label{eq:fEllprops}
\end{eqnarray}
where $w_{h,p}(d_1,d_2):=h(d_1,d_2)f_\infty(p+1-d_1^2d_2,p)$. \inlinelabel{eq:whp}

\subsubsection{The work of David--Koukoulopoulos--Smith}

The computation of weighted sums of Euler products similar to \eqref{eq:fEllprops} is another topic of \cite{DKS17}, more particularly when summing over primes. David, Koukoulopoulos and Smith give general results to evaluate sums of the form
\[\sum_{\bs a\in\Ac}w_{\bs a}\prod_{\ell}(1+\delta_\ell(\bs a)),\]
where $d\ge 1$, $\Ac\subset\Z^d\cap[-X,X]^d$, asymptotically when $X\to\infty$, with $\delta_\ell(\bs a), w_{\bs a}\in\C$ satisfying certain properties, such as:
\begin{enumerate}
\item\label{item:periodicity} For primes $\ell$ small enough (with respect to $X$) and integers $r\ge 1$, there exists $\Delta_{\ell^r}:\Z/\ell^r\to\C$ such that $\delta_{\ell}(\bs a)=\Delta_{\ell^r}(\bs a\pmod{\ell^r})$ for ``most'' $\bs a\in\Ac$;
\item\label{item:welldistributed} The parameter set $\Ac$ is well-distributed modulo small enough $q\ge 2$, with respect to the weights $w_{\bs a}$, i.e.
  \[\sum_{\substack{\bs a\in\Ac\\ \bs a\equiv \bs b\pmod{q}}} w_{\bs a}\approx \frac{1}{q}\sum_{\bs a\in\Ac}w_{\bs a}\]
  for all $\bs b\in(\Z/q)^d$.
\end{enumerate}

In \cite{DKS17}, the set $\Ac$ essentially parametrizes primes in some intervals, and statistics on elliptic curves over $\F_p$ are obtained on average over primes $p\le X$. Condition \ref{item:welldistributed} then amounts to studying the distribution of primes in (short) arithmetic progressions.\\

However, Theorems 4.1 and 4.2 of ibid. are not directly applicable to our situation, since the condition \ref{item:periodicity} above ((4) or (4') in ibid.) does not hold: the matrix density in the limit defining $f_\ell(d_1,d_2,p)$ is periodic with respect to $d_2$, but not with respect to $d_1$. Moreover, the distribution of $h$ in arithmetic progressions is not as simple as in \ref{item:welldistributed}, as we will see.

Nonetheless, we can still use the ideas and some of the results of David--Koukoulopoulos--Smith towards our goal, as we describe in the remaining of this section.\footnote{An alternative approach allowing to use \cite[Theorem 4.2]{DKS17} directly would be to move the factors at primes $\ell\mid d_2$ into the weights $w_{h,p}$. However, essentially the same additional work is then needed to identify the main term as an Euler product and control the errors. Moreover, this would only apply to the $h=s,c$, while Theorem \ref{thm:general} applies to a more general class of functions.}

\subsubsection{Approximate independence of local factors}
We can write \eqref{eq:fEllprops} as
\begin{equation}
  \label{eq:fEllExp}
  h(\Ell(p))=W_{h,p}\cdot\E_{h,p} \left(\prod_{\ell}f_\ell(d_1,d_2,p)\right)  
\end{equation}
with respect to the probability measure on $\{(d_1,d_2) : d_1\mid p-1, \ p_-\le d_1^2d_2\le p_+\}$ induced by the weights $w_{h,p}(d_1,d_2)$, where $W_{h,p}:=\sum_{d_1,d_2}w_{h,p}(d_1,d_2)$ is the normalization factor.

If the local densities $f_\ell(d_1,d_2,p)$ behave independently at each prime $\ell$, then the expectation in \eqref{eq:fEllExp} becomes
\begin{eqnarray}
  \prod_{\ell}\E_{h,p} \left(f_\ell(d_1,d_2,p)\right)&=&\prod_{\ell}\E_{h,p}\left(1+\delta_{\ell}(d_1,d_2,p)\right)\label{eq:eulerProduct}\\
                                                     &=&\sum_{n\ge 1}\mu(n)^2\E_{h,p}(\delta_n(d_1,d_2,p)),\label{eq:independence}
\end{eqnarray}
where $\delta_n(d_1,d_2,p):=\prod_{\substack{\ell\mid n}}(f_\ell(d_1,d_2,p)-1)$. \inlinelabel{eq:deltan}

This independence can be approximated by truncating the Euler product \eqref{eq:eulerProduct}, as in \cite[pp. 37--38]{DKS17}, using a result of Elliott:
\begin{proposition}\label{prop:truncateComplete}
  For any $\varepsilon>0$, $\alpha\ge 1$ and $Z\ge\exp(\sqrt{\log(4p)})$,
  \begin{eqnarray}\label{eq:hEllTruncated}
    h(\Ell(p))&=&W_{h,p}\sum_{\substack{n\ge 1\\ P^+(n)\le z}} \mu(n)^2\E_{h,p}(\delta_n(d_1,d_2,p))\\
              &&+O \left(Z^{\frac{2}{\alpha}+\varepsilon} E_{h,p}^{(B)}+\frac{E_{h,p}^{(G)}}{(\log{Z})^{\alpha-\varepsilon}}\right)\nonumber,
  \end{eqnarray}
  where $z=(\log{Z})^{8\alpha^2}$ and \inlinelabel{eq:z}
  \begin{eqnarray}
    E_{h,p}^{(B)}&:=&\frac{1}{\sqrt{p}}\max \left(\frac{|h(d_1,d_2)|}{d_1} : p_-\le d_1^2d_2\le p_+\right),\label{eq:EB}\\
    E_{h,p}^{(G)}&:=&\frac{1}{p}\sum_{d_1\mid p-1}\int_0^{2\sqrt{p}}\sum_{\frac{p+1-y}{d_1^2}< d_2\le \frac{p+1+y}{d_1^2}} |h(d_1,d_2)|d_py.\label{eq:EG}
  \end{eqnarray}
  The implied constants depend only on $\varepsilon$ and $\alpha$. If there are no Siegel zeros, we may assume that $E_{h,p}^{(B)}=0$.
\end{proposition}
\begin{remark}\label{rem:Bh}
  If $h(d_1,d_2)\ll d_1(d_1d_2)^\varepsilon$ (e.g. for $h=s,c$), then $E_{h,p}^{(B)}\ll p^{-1/2+\varepsilon}$. The second error term will be discussed in more details in Lemma \ref{lemma:EGHyp2}.
\end{remark}
  \subsubsection{Computation of local factors}
  The starting point to compute \eqref{eq:independence} or \eqref{eq:hEllTruncated} is to note that each $f_\ell(d_1,d_2,p)$ is given, according to Proposition \ref{prop:PEH2}, by a limit as $r\to+\infty$ that stabilizes at $r=v_\ell(D_{d_1^2d_2,p})+1$, and that depends only on $v_\ell(d_1)$ and $d_1^2d_2$. Splitting the sums defining the expected values according to $v_\ell(D_{d_1^2d_2,p})$ yields the following:
\begin{proposition}\label{prop:localfactors}
  For $n\ge 1$,
\begin{eqnarray*}
  \E_{h,p}(\delta_n(d_1,d_2,p))&=&\sum_{\substack{q\ge 1\\\rad(q)=n}}\sum_{\substack{a\in\Hc(q)\\d_1\mid p-1}} \Delta_q(a,d_1,p)\frac{\overline w_{h,p}(d_1,a,q)}{W_{h,p}},
\end{eqnarray*}
where $\Hc(q):=\{a\in\Z/q : v_\ell(D_{a,p})=v_\ell(q)-1 \ \forall \ell\mid q\}$, and \inlinelabel{eq:Hcq}
\begin{eqnarray}
  \overline w_{h,p}(d_1,a,q)&:=&\sum_{\substack{\frac{p_-}{d_1^2}< d_2\le \frac{p_+}{d_1^2}\\d_1^2d_2\equiv a\pmod{q}}}w_{h,p}(d_1,d_2),\label{eq:overlinewhp}
\end{eqnarray}
and $\Delta_q: \Z/q\times\N\times\N\to\R$ is a function satisfying $\delta_q(d_1,d_2,p)=\Delta_{q}(d_1^2d_2,d_1,p)$.
\end{proposition}
\begin{remark}
  The map $\Delta_q$ exists (and is unique when evaluated at the parameters considered) since $\delta_q(d_1,d_2,p)$ depends only on $a=d_1^2d_2\pmod{q}$ and $v_\ell(d_1)$.
\end{remark}

\subsubsection{Distribution of $w_{h,p}$ in arithmetic progressions}
To understand $\overline w_{h,p}$, we start by applying Abel's summation formula to the smooth factor $f_\infty$.
\begin{lemma}\label{lemma:Abel}
  \begin{eqnarray*}
    \overline w_{h,p}(d_1,a,q)&=&\frac{1}{\pi p}\int_0^{2\sqrt{p}} \sum_{\substack{\frac{p+1-y}{d_1^2}< d_2\le \frac{p+1+y}{d_1^2}\\ d_1^2d_2\equiv a\pmod{q}}}h(d_1,d_2)d_py,
  \end{eqnarray*}
  where $d_py:=\frac{ydy}{\sqrt{4p-y^2}}$ on $[-2\sqrt{p},2\sqrt{p}]$. \inlinelabel{eq:dpy}
\end{lemma}
For all $q\ge 1$ and $a\in\Z/q$, we may expect that
\begin{eqnarray}
  \label{eq:hd2mod}
  \sum_{\substack{\frac{p+1-y}{d_1^2}< d_2\le \frac{p+1+y}{d_1^2}\\ d_1^2d_2\equiv a\pmod{q}}}h(d_1,d_2)&=&\delta_{(d_1^2,q)\mid a}\frac{2y}{d_1^2}\frac{C_{h,p}(a,d_1,q)}{q}\\
                                                                                                          &&+O \left(\delta_{(d_1^2,q)\mid a} |E_{h,p}(y,d_1,a,q)|\right)\nonumber
\end{eqnarray}
for some $C_{h,p}(a,d_1,q),E_{h,p}(y,d_1,a,q)\in\R$ depending only on the variables in the arguments, assuming at least that the modulus is not too large ($q\le 2y/d_1^2$) and that the interval is large enough ($y\ge d_1^2/2$).
\begin{proposition}\label{prop:hEllTruncatedMain}
  If \eqref{eq:hd2mod} holds, then the main term of \eqref{eq:hEllTruncated} is
  \begin{equation}
    \label{eq:hEllTruncatedMain}
    \sum_{d_1\mid p-1} \frac{1}{d_1^2}\sum_{q\in Q(d_1,z)}\frac{1}{q}\sum_{\substack{a\in\Hc(q)}}\delta_{(q,d_1^2)\mid a} C_{h,p}(a,d_1,q)\Delta_q(a,d_1,p)+O \left(E_{h,p}^{(P)}(z)\right),
  \end{equation}
  where $Q(d_1,z):=\{q\ge 1 : \ell\le z\text{ and } v_\ell(q)>v_\ell(d_1^2)\text{ for all }\ell\mid q\}$ and \inlinelabel{eq:Qd1z}
  \begin{equation}
    \label{eq:EP}
    E_{h,p}^{(P)}(z):=\sum_{\substack{d_1\mid p-1\\ d_1\ll p^{1/2}}}\sum_{\substack{q\in Q(d_1,z)}}\frac{O(1)^{\omega(q)}}{\rad(q)}\frac{1}{p}\int_{0}^{2\sqrt{p}} \sum_{\substack{a\in\Hc(q)\\(q,d_1^2)\mid a}}|E_{h,p}(y,d_1,a,q)|d_py.
  \end{equation}
\end{proposition}

\subsection{The main term as an Euler product}
At this point, one may want to use the multiplicative properties of $\Delta_q$ with respect to $q$ to recover an Euler product from \eqref{eq:hEllTruncatedMain}. However, we need to deal with the additional factor $C_{h,p}(a,d_1,q)$ that may depend on $a$, and that may not be multiplicative in $q$. Note that in \cite{DKS17}, the analogue of this factor depends neither on the class $a$ nor on the modulus $q$.

To overcome this, we now make the additional assumption that when $q\in Q(d_1,z)$, we have
\begin{eqnarray}
  \label{eq:hd2modu}
  C_{h,p}(a,d_1,q)&=&\sum_{\bs v\in\N^m}C_{h,p}^{(1)}(\bs v,d_1)C_{h,p}^{(2)}(a,\bs v,d_1,q)\prod_{\ell\nmid q}C_{h,p}^{(3)}(\bs v,d_1,\ell)\nonumber\\
  |E_{h,p}(y,d_1,a,q)|&\ll&\sum_{\bs v\in\N^m}|C^{(1)}_{h,p}(\bs v,d_1)||E'_{h,p}(y,\bs v,d_1,a,q)|
\end{eqnarray}
for some $m\ge 1$, where $C_{h,p}^{(1)}(\bs v,d_1)$, $C_{h,p}^{(3)}(\bs v,d_1,\ell)$, $E'_{h,p}(y,\bs v,d_1,a,q)\in\R$, the sums and the product have finite support, and $q\mapsto C_{h,p}^{(2)}(a,\bs v,d_1,q)\in\R$ is multiplicative. The multiplicativity of $C^{(2)}_{h,p}$ and $\Delta_q$ then allows to express the main term of \eqref{eq:hEllTruncated} as a weighted sum of (truncated) Euler products:
\begin{proposition}\label{prop:hEllTruncatedMainu}
  If \eqref{eq:hd2modu} holds, then the main term of \eqref{eq:hEllTruncated} is given by
  \begin{eqnarray*}
    &&\sum_{\substack{d_1\mid p-1\\ \bs v\in\N^m}}\frac{C_{h,p}^{(1)}(\bs v,d_1)}{d_1^2}\prod_{\ell} P_{h,p}(\ell,\bs v,d_1),
  \end{eqnarray*}
  with the local factors
  \begin{eqnarray}
    P_{h,p}(\ell,\bs v,d_1)&:=&C_{h,p}^{(3)}(\bs v,d_1,\ell)\label{eq:Php}\\
                           &&+\delta_{\substack{\ell\le z}}\sum_{r> v_\ell(d_1^2)} \frac{1}{\ell^r}\sum_{\substack{a\in\Hc(\ell^r)\\ v_\ell(a)\ge v_\ell(d_1^2)}}  C_{h,p}^{(2)}(a,\bs v,d_1,\ell^r)\Delta_{\ell^r}(a,d_1,p).\nonumber
  \end{eqnarray}
\end{proposition}

\subsection{Computation of the local factors}

Under further natural assumptions, the local factors $P_{h,p}$ can be rewritten as limits of weighted matrix densities:

\begin{proposition}\label{prop:alternativePellud1p}
  For $d_1\mid p-1$ and $\bs v\in\N^m$ fixed, let us assume that for some integer $r_{\ell,\bs v,d_1}\ge v_\ell(d_1^2)$,
  \begin{equation}
    \label{eq:stabilizes}
    r\mapsto C^{(2)}_{h,p}(a,\bs v,d_1,\ell^r)\hspace{0.5cm}
    \begin{cases}
      \text{vanishes if }&v_\ell(d_1^2)\le r\le r_{\ell,\bs v,d_1},\\
      \text{stabilizes as }&r> r_{\ell,\bs v,d_1},
    \end{cases}
  \end{equation}
    \begin{equation}
    \label{eq:dependsvl}
    C^{(2)}_{h,p}(a,\bs v,d_1,\ell^r)\hspace{0.5cm}
    \begin{cases}
      \text{depends only on }&v_\ell(a),\\
      \text{vanishes if }&v_\ell(a)<r_{\ell,\bs v,d_1}.
    \end{cases}
  \end{equation}
  Then, the local factor at $\ell$ is
  \begin{eqnarray*}
    P_{h,p}(\ell,\bs v,d_1)&=&C_{h,p}^{(3)}(\bs v,d_1,\ell)+\delta_{\substack{\ell\le z}}L_{h,p}(\bs v,d_1,\ell,r_{\ell,\bs v,d_1}),
  \end{eqnarray*}
  where for any integers $r\ge 1$ and $v,w\ge 0$ (in \eqref{eq:gpwv}),
  \begin{equation}
    \label{eq:Lhp}
    L_{h,p}(\bs v,d_1,\ell,r):=\delta_{\substack{v_\ell(p-1)\ge \frac{r}{2}}}\lim_{R\to\infty}\sum_{w=r}^{R} C^{(2)}_{h,p}(\ell^w,\bs v,d_1,\ell^R)g_p(w,v_\ell(d_1),\ell^R),
  \end{equation}
  \begin{eqnarray}
    g_p(w,v,\ell^R)&:=&\frac{\left|\left\{g\in M_2(\Z/\ell^R) : \substack{\det(g)=p\\ v_\ell(p+1-\tr(g))=w \\g\equiv 1\pmod{\ell^{v}}\\ g\not\equiv 1\pmod{\ell^{v+1}}}\right\}\right|}{\ell^{3R}(1-1/\ell^2)}-\frac{1-1/\ell}{\ell^{w}}.\label{eq:gpwv}
  \end{eqnarray}
\end{proposition}

\begin{remark}
  Concerning \eqref{eq:dependsvl}, note that it is natural that $C_{h,p}(a,d_1,\ell^r)$ in \eqref{eq:hd2mod} depends only mildly on $a$.
\end{remark}
\begin{example}\label{ex:C2practice}
  In the case $h\in\{s,c\}$, we will have $m=3$, $\bs v=(u,k,i)\in\N^3$ for $i\in\{0,1\}$, $u\mid d_1$, $k\mid d_1^2/u$, and
  \begin{eqnarray}
    \label{eq:C2practice}
    C^{(2)}_{h,p}(a,\bs v,d_1,\ell^r)=
    \begin{cases}
      (d_1^2,\ell^r)&:i=0\\
      (d_1^2,\ell^r)c_k(\ell^{v_\ell(a/d_1^2)})\delta_{r\ge v_\ell(d_1^2k)}&:i=1
    \end{cases}
  \end{eqnarray}
  for $c_k$ the Ramanujan sum modulo $k$ (see Sections \ref{sec:hs}--\ref{sec:hc}). In Proposition \ref{prop:maintermS}, we will see that this implies that \eqref{eq:stabilizes}--\eqref{eq:dependsvl} hold with $r_{\ell,\bs v,d_1}=v_\ell(d_1^2)+\delta_{i=1}\max(0,v_\ell(k)-1)$.
\end{example}

\subsubsection{Computation of the local densities}

Using the computations of Gekeler \cite{Gek03} and David--Koukoulopoulos--Smith \cite{DKS17}, explicit and asymptotic expressions for the densities $g_p(w,v,\ell^r)$ can be given. For the sake of brevity, we give only the latter; for the former, it suffices to combine the proof of the proposition below with \cite[Theorem 4.4]{Gek03}.
\begin{proposition}\label{prop:localDensities}
  For $R$ large enough with respect to $\ell$, $v$ and $w$, we have, assuming $\ell^v\mid p-1$ and $w\ge 2v$,
  \begin{eqnarray*}
    g_p(w,v,\ell^R)&=&\frac{1}{\ell^w}\left(1-\frac{1}{\ell}\right)\left(\frac{1}{\ell^{v}}\oO{\frac{1}{\ell}}-1\right).
  \end{eqnarray*}  
  Moreover, if $\ell\neq p$, $\lim_{R\to\infty}\sum_{w=0}^R g_p(w,0,\ell^R)=-\frac{\delta_{\ell\mid p-1}}{\ell(\ell^2-1)}$.
\end{proposition}

\subsection{Conclusion}
\begin{definition}\label{def:L'}
  For $d_1\mid p-1$, $\bs v\in\N^m$, $r\ge 1$ and $\ell$ prime, we let
  \begin{eqnarray*}
    L'_{h,p}(\bs v,d_1,\ell,r)&:=&\delta_{\substack{v_\ell((p-1)^2)\ge r}}\left(1-\frac{1}{\ell}\right)\lim_{R\to\infty}\sum_{w=r}^R\frac{C_{h,p}^{(2)}(\ell^w,\bs v,d_1,\ell^R)}{\ell^w},\\
    E_{h,p}^{(T)}(z,\bs v,d_1)&:=&\sum_{\ell>z}\left|\frac{L_{h,p}(\bs v,d_1,\ell,r_{\ell,\bs v,d_1})}{C^{(3)}_{h,p}(\bs v,d_1,\ell)}\right|.
  \end{eqnarray*}
\end{definition}
Combining the previous sections together, we obtain:
\begin{theorem}\label{thm:general}
  Let $\varepsilon>0$, $\alpha\ge 1$, and $Z\ge\exp(\sqrt{\log(4p)})$. If \eqref{eq:hd2modu}, \eqref{eq:stabilizes}, \eqref{eq:dependsvl} and
  \begin{equation}\label{eq:ETBound}
    E_{h,p}^{(T)}(z,\bs v,d_1)\le 1\qquad\text{whenever}\qquad C_{h,p}^{(1)}(\bs v,d_1)\neq 0
  \end{equation}
  hold, then
  \begin{eqnarray*}
    h(\Ell(p))&=&\sum_{\substack{d_1\mid p-1\\ \bs v\in\N^m}}\frac{C_{h,p}^{(1)}(\bs v,d_1)}{d_1^2}\prod_{\ell}\left[C_{h,p}^{(3)}(\bs v,d_1,\ell)+L_{h,p}(\bs v,d_1,\ell,r_{\ell,\bs v,d_1})\right]\\
    &&\hspace{3cm}\oO{E_{h,p}^{(T)}(z,\bs v,d_1)}\\
    &&+O \Bigg(Z^{\frac{2}{\alpha}+\varepsilon} E_{h,p}^{(B)}+\frac{E_{h,p}^{(G)}}{(\log{Z})^{\alpha-\varepsilon}}+E_{h,p}^{(P)}\left((\log{Z})^{8\alpha^2}\right)\Bigg)\\
  \end{eqnarray*}
  where $L_{h,p}$, as defined in \eqref{eq:Lhp}, satisfies
  \begin{eqnarray*}    
    L_{h,p}(\bs v,d_1,\ell,r_{\ell,\bs v,d_1})&=&L'_{h,p}(\bs v,d_1,\ell,r_{\ell,\bs v,d_1})\left(\frac{1}{\ell^{v_\ell(d_1)}}\oO{\frac{1}{\ell}}-1\right).
  \end{eqnarray*}
  The implied constants depend only on $\varepsilon$ and $\alpha$.
\end{theorem}

\begin{remark}
  The error terms account for (B)ad conductors and (G)ood conductors (in Proposition \ref{prop:truncateComplete}), distribution of $h$ in arithmetic (P)rogressions (in Proposition \ref{prop:hEllTruncatedMain}), and completion of the (T)runcated products (in Theorem \ref{thm:general}).
\end{remark}

\subsection*{Index of notations}
For convenience, we provide a index of the most important quantities defined above.
\begin{center}
  \small
      \renewcommand{\arraystretch}{1.5}
  \begin{longtable}{llp{5cm}}
  Notation&Reference/page&Description\\ \hline \endhead
  $h(d_1,d_2)$&p. \pageref{page:hd1d2}&$h(\Z/d_1\times\Z/d_1d_2)$\\
  $P(E(\F_p)\cong H)$&\eqpageref{eq:PEH}&\\
  $p_\pm$&p. \pageref{page:qpm}&$p+1\pm 2\sqrt{p}$\\
  $f_\ell$ and $f_\infty$&\eqpageref{eq:fell} and \eqpageref{eq:finfty}&Local factors\\
  $D_{a,p}$&\eqpageref{eq:Dap}&Discriminant\\
  $w_{h,p}(d_1,d_2)$&\eqpageref{eq:whp}&Weights\\
  $\delta_{n}(d_1,d_2,p)$&\eqpageref{eq:deltan}&Local factors\\
  $z=(\log{Z})^{8\alpha^2}$&\eqpageref{eq:z}&Euler product truncation\\
  $E_{h,p}^{(B)}$ and $E_{h,p}^{(G)}$&\eqpageref{eq:EG} and \eqpageref{eq:EB}&Errors in truncating the Euler product\\
  $\Delta_q(a,d_1,p)$&Proposition \ref{prop:localfactors} p. \pageref{prop:localfactors}&$\delta_q(d_1,d_2,p)=\Delta_{q}(d_1^2d_2,d_1,p)$\\
  $\Hc(q)$&\eqpageref{eq:Hcq}&Classes $a\in\Z/q$ verifying conditions on $D_{a,p}$\\
    $\overline w_{h,p}(d_1,a,q)$&\eqpageref{eq:overlinewhp}&$w(d_1,\cdot)$ in short intervals and arithmetic progressions (SIAP)\\
    $d_py$&\eqpageref{eq:dpy}&$ydy/\sqrt{4p-y^2}$\\
    $C_{h,p}(a,d_1,q)$&\eqpageref{eq:hd2mod}&Main term for $h(d_1,\cdot)$ in SIAP\\
    $E_{h,p}(y,d_1,a,q)$&\eqpageref{eq:hd2mod}&Error for $h(d_1,\cdot)$ in SIAP\\
    $Q(d_1,z)$&\eqpageref{eq:Qd1z}&Admissible moduli\\
    $E_{h,p}^{(P)}(z)$&\eqpageref{eq:EP}&Total error for $h$ in SIAP\\
    \begin{minipage}{2cm}
      $C_{h,p}^{(1)}(\bs v,d_1)$,\\
      $C_{h,p}^{(3)}(\bs v,d_1,\ell)$
    \end{minipage}&\eqpageref{eq:hd2modu}&Decomposition of $C_{h,p}$\\
    $C_{h,p}^{(2)}(a,\bs v,d_1,q)$&\eqpageref{eq:hd2modu}&Multiplicative part of $C_{h,p}$\\
    $E'_{h,p}(y,\bs v,d_1,a,q)$&\eqpageref{eq:hd2modu}&Decomposition of $E_{h,p}$\\
    $P_{h,p}(\ell,\bs v,d_1)$&\eqpageref{eq:Php}&Local factors in the main term\\
    $r_{\ell,\bs v,d_1}$&Proposition \ref{prop:alternativePellud1p} p. \pageref{prop:alternativePellud1p}&Threshold\\
    $L_{h,p}(\bs v,d_1,\ell,r)$&\eqpageref{eq:Lhp}&Weighted sum of matrix densities\\
    $g_p(w,v,\ell^R)$&\eqpageref{eq:gpwv}&Matrix density\\
    $L'_{h,p}(\bs v,d_1,\ell,r)$&Definition \ref{def:L'} p. \pageref{def:L'}&Factor of $L_{h,p}$\\
    $E_{h,p}^{(T)}(z,\bs v,d_1)$&Definition \ref{def:L'} p. \pageref{def:L'}&Truncation error
\end{longtable}
\end{center}
\section{Divisor function in arithmetic progressions and short intervals}\label{sec:DeltaqAB}

To apply Theorem \ref{thm:general} to $h\in\{s,c\}$, in order to show that \eqref{eq:hd2mod} holds, we will need to understand
\begin{equation}
  \label{eq:tauIntervalMod}
  \sum_{\substack{A< n\le B\\ n\equiv a\pmod q}} \tau(n),
\end{equation}
where $1\le q\le A$, $a\in\Z/q$ and $[A,B]$ is a short interval (i.e. of length $o(A)$).

The sum \eqref{eq:tauIntervalMod} should be asymptotically equal (under admissible ranges) to $D(B,a,q)-D(A,a,q)$, where
\[D(X,a,q):=\frac{1}{q}\sum_{k\mid q}\frac{c_k(a)}{k} X \left(\log \left(\frac{X}{k^2}\right)+2\gamma-1\right),\]
with $c_k(a)$ the Ramanujan sum modulo $k$ given by
\begin{equation}
  \label{eq:Ramanujan}
  c_k(a)=\sum_{n\in(\Z/k)^\times}e(an/k)=\sum_{f\mid (k,a)}f\mu(k/f).
\end{equation}
  
\begin{definition}
  For $1\le q\le A<B$ and $a\in\Z/q$, we let
  \begin{eqnarray*}
    \Delta(B,a,q)&:=&\sum_{\substack{n\le B\\ n\equiv a\pmod{q}}} \tau(n)-D(B,a,q),\\
    \Delta(B)&:=&\Delta(B,0,1),\\
    \Delta(A,B,a,q)&:=&\Delta(B,a,q)-\Delta(A,a,q)\\
    \Delta(A,B)&:=&\Delta(A,B,0,1).
  \end{eqnarray*}
  For $C> 0$, we will also use the abbreviation $\Delta(A\mp C,a,q)$ for $\Delta(A-C,A+C,a,q)$.
\end{definition}
Combining ideas from \cite{Blo07} (for the arithmetic progression aspect) and \cite{IvZh14} (for the short interval aspect), we will prove the following mean square result on the divisor function simultaneously in arithmetic progressions and short intervals:

\begin{theorem}\label{thm:tauAPSI}
  Let $1\le A<B$, $1\le q\le \sqrt{A}$, and $\varepsilon>0$. We have
  \begin{eqnarray*}
    &&\frac{1}{q}\sum_{a\in\Z/q}|\Delta(A,B,a,q)|^2\\
    &&\hspace{2cm}\ll (qB)^\varepsilon
          \begin{cases}
            \frac{(B-A)^{1/2}}{q}\left(\frac{B^3}{A}\right)^{1/4}&:B-A\le \sqrt{B},\\
            \frac{(B-A)^{4/3}}{q^{4/3}}\left(\frac{B}{A}\right)^{1/3}&:\sqrt{B}\le B-A\le\sqrt{AB}.
          \end{cases}                                                               
        \end{eqnarray*}
\end{theorem}
\begin{remark}
  The range we will need is $B\ll A\ll B$, $B-A\le\sqrt{B}$ and $q\le A^{1/2}$. The first part of theorem then gives that $|\Delta(A,B,a,q)|^2$ is at most $(qB)^\varepsilon \sqrt{A(B-A)}/q$ on average over $a$.
\end{remark}
\subsection{Some existing results}

Up to smoothing the sum, the Voronoi summation formula (see \cite[Section 4]{IK04}) gives an explicit expression for $\Delta(B,a,q)$, from which one readily gets that if $(a,q)=1$, then for any $\varepsilon>0$
\begin{equation}
  \label{eq:Voronoi1}
  \Delta(B,a,q)\ll (qB)^\varepsilon(q^{1/2}+B^{1/3}),
\end{equation}
and Pongsriiam--Vaughan \cite{PongVau15} showed that this also holds when $(a,q)>1$. We also mention their result \cite{PongVau17} on average over $q$ and $a$, improving on Motohashi (here in a weakened form):
\begin{eqnarray*}
  \sum_{q\le Q}\sum_{a\in(\Z/q)^\times} |\Delta(B,a,q)|^2\ll\ B^\varepsilon \left(QB+B^{5/3}+Q^{1+\theta}B^{1-\theta}\right)
\end{eqnarray*}
for $1\le Q< B$ and $\theta\in(0,1)$.

Improving significantly on previous results of Banks--Heath-Brown--Shparlinski \cite{BHBS05}, Blomer \cite{Blo07} showed that for any $\varepsilon>0$
\begin{equation}
  \label{eq:Blomer}
  \sum_{a\in\Z/q}|\Delta(B,a,q)|^2\ll B^{1+\varepsilon},
\end{equation}
which gives a better result on average than \eqref{eq:Voronoi1} if $B\ll_\varepsilon q^{2-\varepsilon}$.

\subsubsection{Short intervals} Let us assume that $q=1$ and let $\varepsilon>0$. Using \eqref{eq:Voronoi1}, we get $|\Delta(A,B)|\ll B^{1/3+\varepsilon}$,
which can be nontrivial only when $B-A\gg B^{1/3}$. The exponent can be reduced to $131/416+\varepsilon$ using the latest result of Huxley on $\Delta(B)$, and conjecturally to $1/4+\varepsilon$, for any $\varepsilon>0$.

Exploiting the short interval aspect, Ivić and Zhai \cite[Section 3]{IvZh14} recently obtained that
  \begin{equation}
    \label{eq:E2beta123}
    |\Delta(A,B)|\ll A^{\nu_1}(B-A)^{\nu_2}\text{ if }1\ll B-A\ll A^{\nu_3}
\end{equation}
when $(\nu_1,\nu_2,\nu_3)=(1/4+\varepsilon,1/4,3/5)$ or $(2/9+\varepsilon,1/3,2/3)$, for any $\varepsilon>0$.

A conjecture of Jutila (see \cite{Jutila84}, \cite[Conjecture 3]{IvZh14}) asserts that \eqref{eq:E2beta123} holds with $(\nu_1,\nu_2,\nu_3)=(\delta, 1/2, 1/2-\delta)$ if $\delta\in(0,1/4)$ and $B-A\ge A^\delta$. This is supported by average results over $A$ for certain ranges: for example
\begin{equation}
  \label{eq:Jutila}
  \int_T^{T+H} |\Delta(A,A+U)|^2dA\ll T^\varepsilon(HU+T)
\end{equation}
if $1\le U\le \sqrt{T}/2\ll H\le T$ (see \cite{Jutila84}, \cite[Theorem 2]{Jutila89}).

\begin{remark}
  With an additional average over $A$ when $[A,B]=[A,A+U]$ with $U$ fixed, Theorem \ref{thm:tauAPSI} could probably be improved as in Jutila's results (see \eqref{eq:Jutila} below). We also mention a recent preprint of Kerr and Shparlinski \cite{KerrShpar18} combining Blomer's technique with bounds on bilinear sums of Kloosterman sums.
\end{remark}
\subsection{Proof of Theorem \ref{thm:tauAPSI}}

We start by a truncated Voronoi formula in arithmetic progressions:
\begin{lemma}\label{lemma:Voronoiq}
  For any $\varepsilon>0$, $X\ge 1$, $q\ge 1$, $a\in\Z/q$ and $N\ge 1$ such that $1/X\le N/q^2\le X$, we have
    \begin{eqnarray*}
      \Delta(X,a,q)&=&\frac{1}{q}\sum_{k\mid q}\sum_{W=Y,K}F_{W}(X,a,k,N)+O \left((Xq)^\varepsilon\left(\left(\frac{X}{N}\right)^{1/2}+1\right)\right),
    \end{eqnarray*}
    and
  \begin{eqnarray*}
    \sum_{a\in\Z/q}\left|\Delta(X,a,q)-\frac{1}{q}\sum_{k\mid q}\sum_{W=Y,K}F_{W}(X,a,k,N)\right|^2&\ll&(Xq)^\varepsilon \frac{qX}{N},
  \end{eqnarray*}
    where
    \begin{eqnarray*}
      F_{W}(X,a,k,N)&:=&\frac{C_W}{k}\sum_{n\le N}\tau(n)\Kl_k(\delta_Wa,n)\int_0^X W_0 \left(\frac{4\pi}{k}\sqrt{nx}\right)dx\\
                  &\ll&(Xq)^\varepsilon \left(\frac{X}{N}\right)^{1/2},\\
      (C_W,\delta_W)&=&\begin{cases}
          (-2\pi,1)&W=Y\\
          (4,-1)&W=K,\text{ and}
        \end{cases}\\
      \Kl_k(c,d)&=&\sum_{x\in(\Z/k)^\times} e \left(\frac{cx+d\overline x}{k}\right) \hspace{1cm} (c,d\in\Z/k),      
    \end{eqnarray*}
    for $W\in\{Y,K\}$, with $Y$ (resp. $K$) the Bessel (resp. modified Bessel) functions of the second kind.
\end{lemma}
\begin{proof}
  By Voronoi's summation formula \cite[(4.49)]{IK04}, if $g\in C^\infty_c(\R_+)$ and $(b,q)=1$, we have
  \begin{eqnarray}
    &&\sum_{n\ge 1}\tau(n) e(bn/q)g(n)\label{eq:voronoi}\\
    &&=\frac{1}{q}\int_0^\infty \left(\log \left(\frac{x}{q^2}\right)+2\gamma\right)g(x)dx\nonumber\\
    &&+\frac{1}{q}\sum_{W=Y,K}C_W\sum_{n\ge 1}\tau(n)e(-\delta_W\overline bn/k)\int_0^\infty W_0 \left(\frac{4\pi}{q}\sqrt{nx}\right)g(x)dx.\nonumber
  \end{eqnarray}
  By the orthogonality relations for $\Z/q$,
  \begin{eqnarray*}
    \sum_{\substack{n\ge 1\\ n\equiv a\pmod{q}}} \tau(n)g(n)&=&\frac{1}{q}\sum_{b\in\Z/q}e(-ab/q)\sum_{n\ge 1}\tau(n) e(bn/q)g(n)\\
                                                            &=&\frac{1}{q}\sum_{k\mid q}\sum_{b\in(\Z/k)^\times}e(-ab/k)\sum_{n\ge 1}\tau(n)e(bn/k)g(n).
  \end{eqnarray*}
  Using \eqref{eq:voronoi}, this is equal to
  \begin{eqnarray*}
    &&\frac{1}{q}\sum_{k\mid q}\frac{c_k(a)}{q}\int_0^\infty \left(\log \left(\frac{x}{q^2}\right)+2\gamma\right)g(x)dx\\
    &&+\frac{1}{q}\sum_{k\mid q}\sum_{W=Y,K}\frac{C_W}{k}\sum_{n\ge 1}\Kl_k(\delta_Wa,n)\tau(n)\int_0^\infty W_0 \left(\frac{4\pi}{k}\sqrt{nx}\right)g(x)dx.
  \end{eqnarray*}
  
  For $0\le X_1\ll X$, let $g=g_{X,X_1}$ be supported on $[0,X+X_1]$, identically equal to $1$ on $[0,X]$, and such that $||g^{(i)}||_\infty\ll 1/X_1^{i}$ for all $i\ge 0$, where $g^{(i)}$ is the $i$th derivative of $g$. Then
  \begin{eqnarray*}
    \Delta(X,a,q)&=&\frac{1}{q}\sum_{k\mid q}\sum_{W=Y,K}\tilde F_{W}(g,a,k)+O \Big((Xq)^\varepsilon \rho(X,X_1,a,q)\Big),
  \end{eqnarray*}
  where $\rho(X,X_1,a,q)=|\{n\in[X,X+X_1]\cap\N : n\equiv a\pmod{q}\}|$, and for $W\in\{Y,K\}$,
    \begin{eqnarray*}
      \tilde F_W(g,a,k)&:=&\frac{C_W}{k}\sum_{n\ge 1}\tau(n)\Kl_k(\delta_Wa,n)\int_0^\infty W_0(4\pi\sqrt{nx}/k)g(x)dx.
    \end{eqnarray*}
    For any $j\ge 1$, we have from the decay of $g$ and $W_0$ that (see \cite[(9-10)]{Blo07})
  \begin{eqnarray*}
    \left|\int_0^\infty W_0 \left(4\pi\sqrt{nx}/k\right)g(x)dx\right|&\ll&\frac{1}{(\sqrt{n}/k)^{j+1/2}}\int_0^\infty x^{j/2-1/4}|g^{(j)}(x)|dx\\
                                                                     &\ll&\frac{k^{j+1/2}}{n^{j/2+1/4}X_1^{j-1}}X^{j/2-1/4}.
  \end{eqnarray*}
  Thus, for any integer $j\ge 2$ and $\varepsilon>0$, using the bound $|\Kl_k(\delta_W a,n,k)|\le\tau(k)(a,n,k)^{1/2}k^{1/2}$, we have
  \begin{eqnarray*}
    \tilde F_W(g,a,k)&=&\frac{C_W}{k}\sum_{n\le N}\tau(n)\Kl_k(\delta_Wa,n)\int_0^X W_0 \left(4\pi\sqrt{nx}/k\right)dx\\
            &&+O \left(\frac{k^{j+1/2+\varepsilon}X^{j/2-1/4}}{X_1^{j-1}N^{j/2-3/4-\varepsilon}}\right).
  \end{eqnarray*}
  Therefore,
  \begin{eqnarray*}
      \Delta(X,a,q)&=&\frac{1}{q}\sum_{k\mid q}\sum_{W=Y,K}F_{W}(X,a,k,N)\\
                   &&+O \left((Xq)^\varepsilon\rho(X,X_1,a,q)+\frac{q^{j+1/2+\varepsilon}X^{j/2-1/4}}{X_1^{j-1}N^{j/2-3/4-\varepsilon}}\right).
  \end{eqnarray*}
  The results follows from taking $j$ large enough and choosing
  \[X_1=\floor{q \left(\frac{X}{N}\right)^{1/2}\left(\frac{qN^{1/2}}{X^{1/2}}\right)^{\frac{3}{2j}}},\]
  which is admissible under the conditions on $N$.
\end{proof}

It follows from Lemma \ref{lemma:Voronoiq} and the Cauchy--Schwarz inequality that for
\begin{equation}
  \label{eq:condN}
 1+q^2/A\le N\le q^2B,
\end{equation}
we have
\begin{equation}
  \label{eq:CS}
  \sum_{a\in\Z/q}|\Delta(A,B,a,q)|^2\ll\frac{1}{q^2}\sum_{k\mid q}\sum_{W=Y,K}\sum_{a\in\Z/q} \left|\left[F_{W}(X,a,k,N)\right]_A^B\right|^2+\frac{(qB)^{1+\varepsilon}}{N}.
\end{equation}
Then, we open the square and exploit cancellation among Kloosterman sums:
\begin{lemma}\label{lemma:FS} For $W\in\{Y,K\}$, $k\mid q$, any $Z\in C([A,B],\R_{>0})$ and $N\ge 1$, we have
  \begin{eqnarray*}
    \sum_{a\in\Z/q}\left|\left[F_{W}(X,a,k,N)\right]_A^B\right|^2&\ll&\frac{q}{k^2}\int_A^B Z(v)^2dv\\
    &&\sum_{f\mid k}f\sum_{\substack{n,n'\le N\\ n\equiv n'\pmod{f}}}\tau(n)\tau(n')|S_{W,Z}(n,n',k)|,
\end{eqnarray*}
\[\text{where}\quad S_{W,Z}(n,n',k):=\int_A^B \frac{1}{Z(v)^2} W_0 \left(\frac{4\pi}{k}\sqrt{nv}\right)W_0 \left(\frac{4\pi}{k}\sqrt{n'v}\right)dv.\]
\end{lemma}
\begin{proof}
By Cauchy--Schwarz, the square $\left|\left[F_{W}(X,a,k,N)\right]_A^B\right|^2$ is
\begin{eqnarray*}
  &\ll&\frac{1}{k^2}\left(\int_A^B Z(v)^2dv\right)\int_A^B\left|\sum_{n\le N}\frac{\tau(n)\Kl_k(\delta_Wa,n)}{Z(v)}W_0 \left(\frac{4\pi}{k}\sqrt{nv}\right)\right|^2dv.
\end{eqnarray*}
The second integral is
\begin{eqnarray*}
  &&\sum_{n,n'\le N}\tau(n)\tau(n')\Kl_k(\delta_Wa,n)\Kl_k(\delta_Wa,n')\\
  &&\hspace{2cm}\int_A^B \frac{1}{Z(v)^2}W_0 \left(\frac{4\pi}{k}\sqrt{nv}\right)W_0 \left(\frac{4\pi}{k}\sqrt{n'v}\right)dv.
\end{eqnarray*}
By orthogonality and the alternative expression \eqref{eq:Ramanujan} for the Ramanujan sum,
\begin{eqnarray*}
  &&\sum_{a\in\Z/q}\Kl_k(\delta_Wa,n)\Kl_k(\delta_Wa,n')\\
  &=&\sum_{x_1,x_2\in(\Z/k)^\times} e\left((n\overline x_1+n'\overline x_2)/k\right)\sum_{a\in\Z/q}e\left(a\delta_W(x_1+x_2)/k\right)\\
  &=&q\sum_{x\in(\Z/k)^\times} e\left(x(n-n')/k\right)=qc_k(n-n')=q\sum_{f\mid k}f\delta_{f\mid n-n'}\mu(k/f).
\end{eqnarray*}
\end{proof}
Finally, we compute the right-hand side of the expression in Lemma \ref{lemma:FS} by approximating the Bessel functions.
\begin{lemma}\label{lemma:aFY} If $q\ll\sqrt{A}$ and $B-A\ll \sqrt{AB}$, then, for $W\in\{Y,K\}$, $k\mid q$, and $N\ge 1$,
  \begin{eqnarray*}
    \sum_{a\in\Z/q}\left|\left[F_{W}(X,a,k,N)\right]_A^B\right|^2&\ll&(kN)^\varepsilon\frac{q(B-A)}{\sqrt{A}}\left[(B-A)\sqrt{N}+N\sqrt{B}\right].
  \end{eqnarray*}
\end{lemma}
\begin{proof}
We have the well-known asymptotic expansions
\begin{eqnarray*}
  Y_0(x)&=&\left(\frac{2}{\pi x}\right)^{1/2}\sin(x-\pi/4)\oO{\frac{1}{x}},\\
  K_0(x)&=&\left(\frac{\pi}{2x}\right)^{1/2}\frac{1}{e^x}\oO{\frac{1}{x}},
\end{eqnarray*}
so that $S_{Y,Z}(n,n',k)$ is equal to
\begin{eqnarray*}
    &&\frac{k}{2(nn')^{1/4}}\int_A^B \left(\cos(4\pi\sqrt{v}(\sqrt{n}-\sqrt{n'})/k)-\sin(4\pi\sqrt{v}(\sqrt{n}+\sqrt{n'})/k)\right)dv\\
    &&\hspace{1cm}+ O \left(\frac{(B-A)k^2}{\sqrt{A}(nn')^{1/4}}\left(\frac{1}{\min(n,n')^{1/2}}+\frac{k}{(nn')^{1/2}\sqrt{A}}\right)\right),
\end{eqnarray*}
taking $Z(v)=v^{-1/4}$. Using that $\int\cos(\lambda\sqrt{x})dx=\frac{2\sqrt{x}\sin(\lambda\sqrt{x})}{\lambda}+\frac{2\cos(\lambda\sqrt{x})}{\lambda^2}$ (and similarly for $\sin$), we get
  \[S_{Y,Z}(n,n',k)\ll\begin{cases}
      \frac{k(B-A)}{\sqrt{n}}\left(1+\frac{k}{\sqrt{A}}\left(\frac{1}{\sqrt{n}}+\frac{k}{n\sqrt{A}}\right)\right),&:n=n'\\
      \frac{k^2}{(nn')^{1/4}}\left(\frac{\sqrt{B}}{\sqrt{n'}-\sqrt{n}}+\frac{B-A}{\sqrt{A}}\left(\frac{1}{\sqrt{n}}+\frac{k}{(nn')^{1/2}\sqrt{A}}\right)\right)&:n<n'.
    \end{cases}
  \]
  If $q\ll\sqrt{A}$, this is
  \[\ll\begin{cases}
      \frac{k(B-A)}{\sqrt{n}},&:n=n'\\
      \frac{k^2}{(nn')^{1/4}}\left(\frac{\sqrt{B}}{\sqrt{n'}-\sqrt{n}}+\frac{B-A}{\sqrt{nA}}\right)&:n<n',
    \end{cases}
  \]
  and 
\begin{eqnarray*}
  &&\sum_{a\in\Z/q}\left|\left[F_{Y}(X,a,k,N)\right]_A^B\right|^2\ll\frac{q}{k}\frac{B-A}{\sqrt{A}}\sum_{f\mid k}f \left[(B-A)\sum_{n\le N}\frac{\tau(n)^2}{n^{1/2}}\right.\\
  &&\left.+\sum_{\substack{n,n'\le N\\ n\equiv n'\pmod{f}\\ n\neq n'}}\frac{\tau(n)\tau(n')k}{(nn')^{1/4}}\left(\frac{\sqrt{B}}{\sqrt{n'}-\sqrt{n}}+\frac{B-A}{\sqrt{A}\min(n,n')^{1/2}}\right)\right]\\
  &\ll&(kN)^\varepsilon\frac{q(B-A)}{\sqrt{A}}\left[(B-A)\sqrt{N}+N\left(\sqrt{B}+\frac{B-A}{\sqrt{A}}\right)\right].
\end{eqnarray*}
If moreover $B-A\ll\sqrt{AB}$, this is
\begin{eqnarray*}
  &\ll&(kN)^\varepsilon\frac{q(B-A)}{\sqrt{A}}\left[(B-A)\sqrt{N}+N\sqrt{B}\right].
\end{eqnarray*}

Similarly, $S_{K,Z}(n,n',k)$ is equal to, taking $Z(v)=v^{-1/4}$ as well,
  \begin{eqnarray*}
    &&\frac{k}{4(nn')^{1/4}}\int_A^B e^{-4\pi(\sqrt{n}+\sqrt{n'})\sqrt{v}/k}\\
    &&\hspace{3cm}\oO{\frac{k}{\sqrt{v}\min(n,n')^{1/2}}+\frac{k^2}{v\sqrt{nn'}}}dv\\
    &\ll&\frac{k}{(nn')^{1/4}}\int_A^B e^{-4\pi(\sqrt{n}+\sqrt{n'})\sqrt{v}/k} dv\\
    &&\hspace{3cm}+ \frac{k^2\sqrt{B}}{(nn')^{1/4}\min(n,n')^{1/2}\sqrt{A}}+\frac{k^3(B-A)}{A(nn')^{3/4}}.
  \end{eqnarray*}
  If $q\le\sqrt{A}$, then the first summand is
  \begin{eqnarray*}
    &\ll&\frac{k^3(B-A)}{(nn')^{1/4}(n+n')A}.
  \end{eqnarray*}
  In this case, by Lemma \ref{lemma:FS},
  \begin{eqnarray*}
    \sum_{a\in\Z/q}\left|\left[F_{K}(X,a,k,N)\right]_A^B\right|^2&\ll&(kN)^\varepsilon \frac{q(B-A)}{\sqrt{A}}\left((B-A)\frac{\sqrt{N}k}{\sqrt{A}}+N\sqrt{B}\right).
  \end{eqnarray*}
\end{proof}
\begin{proof}[Proof of Theorem \ref{thm:tauAPSI}]
  By \eqref{eq:CS} and Lemma \ref{lemma:aFY}, if $q\le \sqrt{A}$ and $B-A\ll\sqrt{AB}$, then for any $N\ge 1$ satisfying \eqref{eq:condN},
\begin{eqnarray*}
  \sum_{a\in\Z/q}|\Delta(A,B,a,q)|^2&\ll&(qB)^\varepsilon \left(\frac{B-A}{q\sqrt{A}}\left[(B-A)\sqrt{N}+N\sqrt{B}\right]+\frac{qB}{N}\right).
\end{eqnarray*}
If $B-A\le\sqrt{B}$, then this is
\begin{eqnarray*}
  &\ll&(qB)^\varepsilon \left(\frac{B-A}{q}\left(\frac{B}{A}\right)^{1/2}N+\frac{qB}{N}\right),
\end{eqnarray*}
and we choose $N=\floor{q \left(\frac{(BA)^{1/2}}{B-A}\right)^{1/2}}$, which satisfies \eqref{eq:condN}. Similarly, if $B-A\ge\sqrt{B}$, we choose $N=\floor{\left(\frac{q^4B^2A}{(B-A)^4}\right)^{1/3}}$.
\end{proof}
\section{Proofs of the results from Section \ref{sec:strategy}}\label{sec:proofs}

\begin{proof}[Proof of Proposition \ref{prop:fellqExplicit}]
  This follows immediately from the properties recalled in Section \ref{subsec:orderMag} along with the fact that $d_1\mid p-1$ by the Weil pairing (see \cite[III.8]{Silv09}).
\end{proof}

\begin{proof}[Proof of Proposition \ref{prop:PEH2}]
  The Proposition is fully contained in \cite[Theorem 3.2]{DKS17}, except \ref{item:PEH2p} that we now check. If $p\nmid d_1^2d_2-1$, then the limit stabilizes at $r=1$, and
  \begin{eqnarray*}
    f_p(d_1,d_2,p)&=&\frac{|\{g\in M_2(\F_p) : \tr(g)=1-d_1^2d_2, \ \det(g)=0\}|}{p^2-1}\\
                  &=&\frac{2(1\cdot p+(p-1)\cdot 1)+(p-2)(p-1)}{p^2-1}
  \end{eqnarray*}
  If $p\mid d_1^2d_2-1$, then $d_1^2d_2=p+1$, $E$ is supersingular, the limit stabilizes at $r=2$, and
  \begin{eqnarray*}
    f_p(d_1,d_2,p)&=&\frac{|\{g\in M_2(\Z/p^2) : \tr(g)=\det(g)=0\}|}{p^4-p^2}\\
                  &=&\frac{\varphi(p^2)^2+p(\varphi(p^2)+\varphi(p)p)}{p^4-p^2}=1.
  \end{eqnarray*}

\end{proof}

\subsection{Truncating the Euler product and proof of Proposition \ref{prop:truncateComplete}}

  \begin{proposition}\label{prop:truncateEP}
  If $\alpha\ge 1$, $\varepsilon,\delta>0$, and $\log{Z}\ge\sqrt{\log(4p)}$, then
    \begin{eqnarray*}
      h(\Ell(p))&=&W_{h,p}\sum_{\substack{n\ge 1\\ P^+(n)\le (\log{Z})^{8\alpha^2}}} \mu(n)^2\E_{h,p}(\delta_n(d_1,d_2,p))\\
            &&+ O \left(Z^{O(1)\delta+\frac{2}{\alpha}}\max_{c\ge 2}\sum_{\substack{d_1,d_2\\\cond(\chi_{d_1,d_2,p})=c}} \frac{|w_{h,p}(d_1,d_2)|}{d_1}\right)\\
                &&+O \left(\frac{1}{(\log{Z})^{\alpha-\varepsilon}}\sum_{\substack{d_1,d_2}} |w_{h,p}(d_1,d_2)|\right).
    \end{eqnarray*}
    The implied constants depend only on $\alpha$, $\delta$ and $\varepsilon$.
  \end{proposition}
  \begin{proof}
    This is similar to \cite[pp. 37--38]{DKS17}, but we spell out the argument because we need different expressions for the errors.

    Let $z=(\log{Z})^{8\alpha^2}$. Using Proposition \ref{prop:PEH2}, we have
\begin{eqnarray*}
  &&\log\prod_{\ell>z} (1+\delta_\ell(d_1,d_2,p))=\sum_{\ell>z} \delta_\ell(d_1,d_2,p)+ O \left(\frac{1}{z\log{z}}\right)\\
  &=&\sum_{\substack{\ell>z\\ \ell\nmid D_{d_1^2d_2,p}/d_1^2}} \frac{\chi_{d_1,d_2,p}(\ell)}{\ell}+\sum_{\substack{\ell>z\\ \ell\mid D_{d_1^2d_2,p}/d_1^2}} \delta_\ell(d_1,d_2,p)+ O \left(\frac{1}{z\log{z}}\right)\\
  &=&\sum_{\ell>z} \frac{\chi_{d_1,d_2,p}(\ell)}{\ell}+O \left(\frac{1}{z\log{z}}+\frac{\omega(D_{d_1^2d_2,p}/d_1^2)}{z}\right).
\end{eqnarray*}

By \cite[Lemma 6.1]{DKS17} (a result going back to Elliott), there exists a set $\Ec_{\alpha}(Z)\subset[1,Z]\cap\Z$ of ``bad conductors'' of size $|\Ec_\alpha(Z)|\le Z^{2/\alpha}$ such that if $\chi$ is a Dirichlet character modulo $d\le \exp((\log{Z})^2)$ with $\cond(\chi)\not\in \Ec_\alpha(Z)$, then
\[ \prod_{\ell>z}\left(1-\frac{\chi(\ell)}{\ell}\right)^{-1}=1+ O \left(\frac{1}{z^{\frac{1}{8\alpha}}}\right),\text{ so that }\sum_{\ell>z}\frac{\chi(\ell)}{\ell}\ll\frac{1}{z\log{z}}+\frac{1}{z^{\frac{1}{8\alpha}}}.\]

By hypothesis, $|D_{d_1^2d_2,p}/d_1^2|\le 4p/d_1^2\le \exp((\log{Z})^2)$. Thus, if $\cond(\chi_{d_1,d_2,p})\not\in \Ec_\alpha(Z)$, then
\begin{eqnarray*}
  \prod_{\ell}(1+\delta_\ell(d_1,d_2,p))&=&\prod_{\ell\le z}(1+\delta_\ell(d_1,d_2,p))\\
                                                 &&+ O \left(\frac{(\log{z})^{O(1)}}{\prod_{\ell\le z}\ell^{v_\ell(d_1)}}\left(\frac{1}{z^{\frac{1}{8\alpha}}}+\frac{\log{p}}{z\log_2{p}}\right)\right)
\end{eqnarray*}
since, by Proposition \ref{prop:PEH2},
\begin{eqnarray*}
  \prod_{\ell\le z}(1+\delta_\ell(d_1,d_2,p))&=&\prod_{\ell\le z}\frac{1}{\ell^{v_\ell(d_1)}}\prod_{\ell\le z}\oO{\frac{1}{\ell}}\ll\frac{(\log{z})^{O(1)}}{\prod_{\ell\le z}\ell^{v_\ell(d_1)}}.
\end{eqnarray*}

On the other hand, if $\cond(\chi_{d_1,d_2,p})\in \Ec_\alpha(Z)$, let us write
\begin{eqnarray*}
  \prod_{\ell}(1+\delta_\ell(d_1,d_2,p))&=&\prod_{\ell\le z_1}(1+\delta_\ell(d_1,d_2,p))\prod_{\ell>z_1}(1+\delta_\ell(d_1,d_2,p)).
\end{eqnarray*}
for $z_1\ge\exp(Z^\delta)\ge \exp(\cond(\chi_{d_1,d_2,p})^\delta)$. As above, we get that this is
\begin{eqnarray*}
  \prod_{\ell}(1+\delta_\ell(d_1,d_2,p))&\ll&\frac{(\log{z_1})^{O(1)}}{\prod_{\ell\le z_1}\ell^{v_\ell(d_1)}}\exp\left(O(1)+\frac{\log{p}}{z_1\log_2{p}}\right),
\end{eqnarray*}
using that $\sum_{\ell>z_1}\frac{\chi_{d_1,d_2,p}(\ell)}{\ell}\ll_\delta 1$, with the implied constant depending only on $\delta$.

If $p$ is large enough, then $z_1\ge p$ and we get by \eqref{eq:fEllprops} that
\begin{eqnarray*}
  h(\Ell(p))&=&\sum_{\substack{d_1,d_2}} w_{h,p}(d_1,d_2)\prod_{\ell\le z}(1+\delta_\ell(d_1,d_2,p))\\
            &&+ O \left(Z^{O(1)\delta}\sum_{\substack{d_1,d_2\\\cond(\chi_{d_1,d_2,p})\in\Ec_\alpha(Z)}} \frac{|w_{h,p}(d_1,d_2)|}{d_1}\right)\\
            &&+O \left(\left(\frac{1}{z^{\frac{1}{8\alpha}-\varepsilon}}+\frac{\log{p}}{z^{1-\varepsilon}\log_2{p}}\right)\sum_{\substack{d_1,d_2}}|w_{h,p}(d_1,d_2)|\right),
\end{eqnarray*}
giving the desired expression.
\end{proof}

\subsubsection{Error terms in the truncation}

\begin{lemma}\label{lemma:conductorsContributions}
  The error terms in Proposition \ref{prop:truncateEP} are respectively bounded by $Z^{\frac{2}{\alpha}+\varepsilon} E_{h,p}^{(B)}$ and $E_{h,p}^{(G)}/(\log{Z})^{\alpha-\varepsilon}$, where $E_{h,p}^{(B)}$ and $E_{h,p}^{(G)}$ are as defined in Proposition \ref{prop:truncateComplete}.
\end{lemma}
  \begin{proof}
    The bound for characters of good conductors follows from Abel's summation formula (see Lemma \ref{lemma:Abel}).
    
    For characters of bad conductors, we start by noting that if $\cond(\chi_{d_1,d_2,p})=c$, then $|D_{d_1^2d_2,p}|=d_1^2cy^2$ for some $u\ge 1$, i.e.
\begin{eqnarray*}
  (p+1-d_1^2d_2)^2+(d_1u)^2c&=&4p,\\
  \text{i.e. }y^2+(d_1u)^2c&=&4p \hspace{0.5cm}(y=p+1-d_1^2d_2),\\
  \text{i.e. }y^2+x^2c&=&4p \hspace{0.5cm} (d_1\mid x).
\end{eqnarray*}
The number of solutions $x,y\in\Z$ to this last diophantine equation is $\ll 1$ (with no dependency on $p$). Indeed, the number of representations of $4p$ by nonequivalent primitive positive-definite binary quadratic forms, up to the $\le 6$ automorphisms, is $\sum_{m\mid 4p} \legendre{-4c_p}{m}\le \tau(4p)\le 6$. Hence, the number of possible values for $u, d_1,d_2$ is $\ll 1$, and
\begin{eqnarray*}
  \sum_{\substack{d_1,d_2\\\cond(\chi_{d_1,d_2,p})=c}} \frac{|w_{h,p}(d_1,d_2)|}{d_1}&\ll&\frac{1}{p}\sum_{u\ge 1}\sum_{d_1\mid p-1}\sum_{\substack{\frac{p_-}{d_1^2}< d_2\le \frac{p_+}{d_1^2}\\ |D_{d_1^2d_2,p}|=d_1^2cu^2}}\sqrt{|D_{d_1^2d_2,p}|}\frac{|h(d_1,d_2)|}{d_1}\\
  &\ll& \frac{1}{\sqrt{p}}\max_{p_-\le d_1^2d_2\le p_+}\frac{|h(d_1,d_2)|}{d_1}.
\end{eqnarray*}
\end{proof}

\subsubsection{Proof of Proposition \ref{prop:truncateComplete}}
The latter now follows from Proposition \ref{prop:truncateEP} and Lemma \ref{lemma:conductorsContributions}.\qed

\subsection{Computation of the main term in \eqref{eq:hEllTruncated} and proofs of Propositions \ref{prop:localfactors} and \ref{prop:hEllTruncatedMain}}

\subsubsection{Preliminary lemmas}
\begin{lemma}\label{lemma:int} For $\alpha>0$, we have
  \begin{eqnarray*}
    \frac{1}{p}\int_{0}^{2\sqrt{p}} y^{\alpha}d_py&=&\frac{\sqrt{\pi}\Gamma(\alpha/2+2)2^{\alpha}}{2(\alpha+2)\Gamma((\alpha+3)/2)}p^{\frac{\alpha-1}{2}}\ll_\alpha p^{\frac{\alpha-1}{2}}.
  \end{eqnarray*}      
\end{lemma}
\begin{lemma}\label{lemma:Hq}
  For $q\ge 2$, we have $|\Hc(q)|\le \sqrt{q\rad(q)}$. Moreover, if $\Hc(q)$ is not empty, then $\rad(q),q/\rad(q)\ll p$.
\end{lemma}
\begin{proof}
  We have $|\Hc(\ell^r)|\le \ell\cdot |\{a\in\Z/\ell^{r-1} : a^2\equiv 4p\}|\le \ell\cdot \ell^{\floor{\frac{r-1}{2}}}$ (see e.g. \cite[Lemma 10]{KaPe17} for the second inequality), so $|\Hc(q)|\le\rad(q)\sqrt{q/\rad(q)}$. The last statement follows from the fact that $|D_{a,p}|\le 4p$.
\end{proof}

\subsubsection{Proof of Proposition \ref{prop:localfactors}}
 The expected value is given by
\begin{eqnarray*}
  &&\sum_{\substack{q\ge 1\\\rad(q)=n}} \sum_{d_1,d_2}\frac{w_{h,p}(d_1,d_2)}{W_{h,p}}\delta_q(d_1,d_2,p)\prod_{\ell\mid q}\delta_{\substack{v_\ell(D_{d_1^2d_2,p})=v_\ell(q)-1}}.
\end{eqnarray*}
The limit over $r$ defining $\delta_\ell(d_1,d_2,p)$ then stabilizes at $r=v_\ell(q)$ and the above is
\begin{eqnarray*}
  &=&\sum_{\substack{q\ge 1\\\rad(q)=n}} \sum_{d_1,d_2}\frac{w_{h,p}(d_1,d_2)}{W_{h,p}}\Delta_q(d_1^2d_2,d_1,p)\prod_{\ell\mid q}\delta_{\substack{v_\ell(D_{d_1^2d_2,p})=v_\ell(q)-1}}\\
  &=&\sum_{\substack{q\ge 1\\\rad(q)=n}} \sum_{d_1\mid p-1}\sum_{a\in\Hc(q)}\Delta_q(a,d_1,p)\sum_{d_2}\frac{w_{h,p}(d_1,d_2)}{W_{h,p}}\delta_{d_1^2d_2\equiv a\pmod{q}}.
\end{eqnarray*}
\qed
\subsubsection{Proof of Proposition \ref{prop:hEllTruncatedMain}}

First, we note that if there exists $a\in\Hc(q)$ with $(d_1^2,q)\mid a$, then $v_\ell(q)>v_\ell(d_1^2)$ for all $\ell\mid q$. Indeed, assume that $r=v_\ell(q)\le v_\ell(d_1^2)$ and that $a$ is as described. Since $D_{a,p}=(p-1)^2+a(2(p+1)-a)$ and $v_\ell(a)\ge r$, we have $v_\ell(d_1^2)\le v_\ell((p-1)^2)=r-1<r\le v_\ell(d_1^2)$, a contradiction. Thus, we can assume that $q\in Q(d_1,z)$.

By Lemma \ref{lemma:Abel} and \eqref{eq:hd2mod}, $\overline w_{h,p}(d_1,a,q)$ is given by
\[\delta_{(d_1^2,q)\mid a}\left[\frac{1}{d_1^2} \frac{C_{h,p}(a,d_1,q)}{q}+O \left(\frac{1}{p}\int_0^{2\sqrt{p}} |E_{h,p}(y,d_1,a,q)|d_py\right)\right].\]
  
  To estimate the total error, we note that, by Proposition \ref{prop:PEH2},
\begin{eqnarray*}
  |\Delta_q(a,d_1,p)|&=&\prod_{\substack{\ell\mid q\\ \ell\nmid d_1}}O \left(\frac{1}{\ell}\right)\prod_{\substack{\ell\mid (q,d_1)}}\left(1+O \left(\frac{1}{\ell}\right)\right)\\
  &\ll&\frac{O(1)^{\omega(q)}}{\rad(q/(q,d_1))}=\frac{O(1)^{\omega(q)}}{\rad(q)},
\end{eqnarray*}
since $v_\ell(q)>v_\ell(d_1^2)$.

Hence, by Proposition \ref{prop:localfactors}, the main term of \eqref{eq:hEllTruncated} is as claimed.
\qed

\subsection{The main term as an Euler product and proofs of Propositions \ref{prop:hEllTruncatedMainu}, \ref{prop:alternativePellud1p} and \ref{prop:localDensities}}

\subsubsection{Proof of Proposition \ref{prop:hEllTruncatedMainu}} This is clear. \qed

  \subsubsection{Proof of Proposition \ref{prop:alternativePellud1p}}
  If \eqref{eq:stabilizes} holds, then
  \begin{eqnarray*}
    &&\sum_{r> v_\ell(d_1^2)} \frac{1}{\ell^r}\sum_{\substack{a\in\Hc(\ell^r)\\ v_\ell(a)\ge v_\ell(d_1^2)}}  C_{h,p}^{(2)}(a,\bs v,d_1,\ell^r)\Delta_{\ell^r}(a,d_1,p)\\
    &=&\lim_{R\to\infty}\sum_{r_{\ell,\bs v,d_1}<r\le R} \frac{1}{\ell^r}\sum_{\substack{a\in\Z/\ell^r\\ v_\ell(D_{a,p})=r-1\\v_\ell(a)\ge v_\ell(d_1^2)}}  C_{h,p}^{(2)}(a,\bs v,d_1,\ell^R)\Delta_{\ell^R}(a,d_1,p)\\
      &=&\lim_{R\to\infty}\frac{1}{\ell^R}\sum_{\substack{a\in\Z/\ell^R\\ r_{\ell,\bs v,d_1}\le v_\ell(D_{a,p})<R\\v_\ell(a)\ge v_\ell(d_1^2)}}  C_{h,p}^{(2)}(a,\bs v,d_1,\ell^R)\Delta_{\ell^R}(a,d_1,p)\\
      &=&\lim_{R\to\infty}\frac{1}{\ell^R}\sum_{\substack{a\in\Z/\ell^R\\ v_\ell(D_{a,p})\ge r_{\ell,\bs v,d_1}\\v_\ell(a)\ge v_\ell(d_1^2)}}  C_{h,p}^{(2)}(a,\bs v,d_1,\ell^R)\Delta_{\ell^R}(a,d_1,p),
    \end{eqnarray*}
    where the last equality follows from the fact that the number of solutions to $(p+1-a)^2\equiv 4p\pmod{\ell^R}$ in $a$ is $\ll\ell^{R/2}$ (see also the proof of Lemma \ref{lemma:Hq}) and $\max_{a\in\Z/\ell^{R}}|C_{h,p}^{(2)}(a,\bs v,d_1,\ell^R)\Delta_{\ell^R}(a,d_1,p)|$ is bounded independently from $R$.

    If \eqref{eq:dependsvl} holds, then this is
    \begin{eqnarray*}
      &&\delta_{\substack{v_\ell((p-1)^2) \ge r_{\ell,\bs v,d_1}}}\lim_{R\to\infty}\frac{1}{\ell^R}\sum_{\substack{a\in\Z/\ell^R\\v_\ell(a)\ge r_{\ell,\bs v,d_1}}}  C_{h,p}^{(2)}(a,\bs v,d_1,\ell^R)\Delta_{\ell^R}(a,d_1,p),
    \end{eqnarray*}
    using that $D_{a,p}=(p-1)^2+a(2(p-1)-a)$. The limit is
    \begin{eqnarray*}
      &&\lim_{R\to\infty}\frac{1}{\ell^R}\sum_{w=r_{\ell,\bs v,d_1}}^R  C_{h,p}^{(2)}(\ell^w,\bs v,d_1,\ell^R)\sum_{b\in(\Z/\ell^{R-w})^\times}\Delta_{\ell^R}(\ell^wb,d_1,p)\\
      &=&\lim_{R\to\infty}\frac{1}{\ell^{3R}}\sum_{w=r_{\ell,\bs v,d_1}}^R  C_{h,p}^{(2)}(\ell^w,\bs v,d_1,\ell^R)\frac{\left|\left\{g\in M_2(\Z/\ell^R) : \substack{\det(g)=p\\ v_\ell(p+1-\tr(g))=w \\g\equiv 1\pmod{\ell^{v}}\\ g\not\equiv 1\pmod{\ell^{v+1}}}\right\}\right|}{1-1/\ell^2}\\
      &&\hspace{2cm}- \left(1-\frac{1}{\ell}\right)\lim_{R\to\infty}\sum_{w=r_{\ell,\bs v,d_1}}^R  \frac{C_{h,p}^{(2)}(\ell^w,\bs v,d_1,\ell^R)}{\ell^w}.
    \end{eqnarray*}
  \qed

\subsubsection{Proof of Proposition \ref{prop:localDensities}}

    \begin{enumerate}[leftmargin=*]
    \item According to \cite[Lemma 3.2(d)]{DKS17}, with $C(\cdots)$ defined in \cite[(3.2)]{DKS17} and $v=v_\ell(d_1)$, the matrix density $g_p(w,v, \ell^R)$ is given by
      \begin{eqnarray*}
        &&\sum_{i=0}^1\sum_{a\in(\Z/\ell^{R-w})^\times}(-1)^i\frac{|C(p+1-a\ell^w,p,\ell^{v+i},\ell^R)|}{\ell^{3R}(1-1/\ell^2)}-\frac{1-1/\ell}{\ell^w}\\
        &=&\frac{1}{\ell^w}\left(\frac{1}{\ell^{v_\ell(d_1)}}\left(1+O \left(\frac{1}{\ell}\right)\right)-1+\frac{1}{\ell}\right).
      \end{eqnarray*}
    \item
If $\ell\neq p$, then $\sum_{w=0}^R g_p(w,0,\ell^R)$ is given by
  \begin{eqnarray*}
    &&1-\frac{1}{\ell^{3R}(1-1/\ell^2)}\left|\left\{g\in M_2(\Z/\ell^R) : \substack{\det(g)=p\\ g\equiv 1\pmod{\ell}}\right\}\right|-\left(1-\frac{1}{\ell^{R+1}}\right)\\
    &=&-\frac{\delta_{\ell\mid p-1}}{\ell(\ell^2-1)}+\frac{1}{\ell^{R+1}}
  \end{eqnarray*}
  by \cite[Lemma 2]{LoWa07} and Hensel's Lemma.
\end{enumerate}
\qed

  \subsection{Bounding the errors}
  Finally, we give some general bounds on the errors appearing in Theorem \ref{thm:general} that will be useful later on.
  \subsubsection{Bound on $E_{h,p}^{(G)}$}
  Applying \eqref{eq:hd2modu} with $q=1$, the following is clear:
\begin{lemma}\label{lemma:EGHyp2}
  If \eqref{eq:hd2modu} holds, then $E_{h,p}^{(G)}$ is
  \[\ll \sum_{\substack{d_1\mid p-1 \\ \bs v\in\N^m}}|C^{(1)}_{h,p}(\bs v,d_1)| \left[\frac{\prod_{\ell}|C_{h,p}^{(3)}(\bs v,d_1,\ell)|}{d_1^2}+\frac{\int_0^{2\sqrt{p}}|E'_{h,p}(y,\bs v,d_1,0,1)|d_py}{p}\right].\]
\end{lemma}
  \subsubsection{Bound on $E_{h,p}^{(P)}$}

  \begin{definition}\label{def:EPaux}
      For $\bs\mu\in[0,1]^3$, $\bs\nu\in[0,1]^5$, we let
    \begin{eqnarray*}
      G_{h,p}(\bs\mu)&:=&\frac{1}{p^{\mu_1}}\sum_{\substack{d_1\mid p-1\\ d_1\ll p^{\mu_2}}}\frac{1}{d_1^{\mu_3}}\max_{\frac{p_-}{d_1^2}< d_2\le \frac{p_+}{d_1^2}}|h(d_1,d_2)|,\\
      F_{h,p}(z,\bs\nu)&:=&\frac{1}{p^{\nu_1}}\sum_{\substack{d_1\mid p-1\\ d_1\ll p^{\nu_2}}}\frac{1}{d_1^{\nu_3}}\sum_{\substack{n\ge 1\\ P^+(n)\le z}}\frac{\mu(n)^2O(1)^{\omega(n)}}{n^{\nu_4}}\\
      &&\hspace{2cm}\sum_{\substack{q\ge 1\\\rad(q)\mid n}}\frac{\max_{a\in\Z/qn}|C_{h,p}(a,d_1,qn)|}{q^{\nu_5}(qn,d_1^2)^{1/2}},
    \end{eqnarray*}
  and for $\delta\in(0,1]$, $\bs\beta\in[0,1/2]\times[0,1]^3$, we let
\begin{eqnarray*}    
  E_{h,p}^{(S)}(z,\delta)&:=&\sum_{\substack{d_1\mid p-1\\ d_1\ll p^{1/4}}}\frac{1}{p}\int_{d_1^2/2}^{2\sqrt{p}} \sum_{\substack{q\le (2y/d_1^2)^\delta\\ q\in Q(d_1,z)}}\frac{O(1)^{\omega(q)}}{\rad(q)}\sum_{\substack{a\in\Hc(q)\\(d_1^2,q)\mid a}}|E_{h,p}(y,d_1,a,q)|d_py,
  \end{eqnarray*}
  \begin{eqnarray*}  
      E_{h,p}^{(L1)}(\beta_4)&:=&G_{h,p} \left(\frac{1}{2}, \frac{1}{4},0\right)+G_{h,p} \left(\frac{1-\beta_4}{2}, \frac{1}{2},2\beta_4\right),\\      
  E_{h,p}^{(L2)}(z,\bs\beta,\delta)&:=&\sum_{i=1}^2F_{h,p}\left(z,\frac{\beta_i\delta}{4},\frac{1}{4},2-\beta_i\delta,1-\delta_{i=2}\beta_2,\frac{1}{2}-\delta_{i=1}\beta_1\right)\\
                             &&+F_{h,p}\left(z,\frac{1-\beta_3}{2},\frac{1}{2},2\beta_3,1,\frac{1}{2}\right).
\end{eqnarray*}
\end{definition}
  \begin{proposition}\label{prop:splitEM}
    For any $\varepsilon>0$, $\delta\in(0,1]$ and $\bs\beta\in[0,1/2]\times[0,1]^3$, we have
  \begin{eqnarray*}
    E_{h,p}^{(P)}(z)&\ll&p^\varepsilon \left((\log{z})^{O(1)}E_{h,p}^{(L1)}(\beta_4)+E_{h,p}^{(L2)}(z,\bs\beta,\delta)\right)+E_{h,p}^{(S)}(z,\delta).
  \end{eqnarray*}
\end{proposition}
\begin{remark}\label{rem:L1h}
  If $h(d_1,d_2)\ll d_1(d_1d_2)^\varepsilon$, then $E_{h,p}^{(L1)}(1/2)\ll 1/p^{1/4-\varepsilon}$. The most delicate contribution to understand will be $E^{(S)}_{h,p}$, which is a sum of the error terms for $h(d_1,\cdot)$ in short intervals and arithmetic progressions (both of admissible sizes).
\end{remark}

\begin{proof}[Proof of Proposition \ref{prop:splitEM}]
We split $E^{(P)}_{h,p}$ into three parts, according to the ranges of $y,q$ and $d_1$:
\begin{enumerate}
\item\label{item:PPartsS} Small enough moduli and large enough intervals: $q\le (2y/d_1^2)^\delta$;
\item\label{item:PPartsLL} Large moduli and large enough intervals: $q> (2y/d_1^2)^\delta$ and $2y/d_1^2> 1$;
\item\label{item:PPartsSI} Small intervals: $y\le d_1^2/2$.
\end{enumerate}

The contribution of the range \ref{item:PPartsS} gives $E_{h,p}^{(S)}$, since $d_1\ll y^{1/2}\ll p^{1/4}$. The contribution of the range \ref{item:PPartsLL} is
\begin{eqnarray*}
  &\ll&\sum_{\substack{d_1\mid p-1\\ d_1\ll p^{1/4}}}\frac{1}{p}\int_{d_1^2/2}^{2\sqrt{p}}\sum_{\substack{q\ge (2y/d_1^2)^\delta\\ P^+(q)\le z\\ \rad(q),\frac{q}{\rad(q)}\ll p}}\frac{c^{\omega(q)}}{\rad(q)}\\
  &&\left[\max_{\frac{p+1-y}{d_1^2}< d_2\le \frac{p+1+y}{d_1^2}}|h(d_1,d_2)|+\frac{y}{d_1^2q}\sum_{\substack{a\in\Hc(q)\\ (d_1^2,q)\mid a}}|C_{h,p}(a,d_1,q)|\right]d_py
\end{eqnarray*}
for some constant $c\ge 1$, using Lemma \ref{lemma:Hq}. Since
\[\sum_{\substack{q\ll p^2\\ P^+(q)\le z}}\frac{c^{\omega(q)}}{\rad(q)}\ll \sum_{\substack{n\ge 1\\ P^+(n)\le z}}\frac{\mu(n)^2}{n}\sum_{\substack{q\ll p^2\\ \rad(q)=n}}1\ll (\log{z})^{O(1)}\exp(\sqrt{\log{p}})^{O(1)},\]
the first summand yields the first part of $E_{h,p}^{(L1)}$. Then, we start noting that, by Lemma \ref{lemma:Hq},
  \begin{eqnarray*}
    |\{a\in\Hc(q) : (d_1^2,q)\mid a\}|&\ll&\left(\frac{q}{(d_1^2,q)}\rad(q)\right)^{1/2}.
  \end{eqnarray*}
  Thus, the second summand gives a contribution of
\begin{eqnarray*}
  &\ll&\sum_{\substack{d_1\mid p-1\\ d_1\ll p^{1/4}}} \frac{1}{p}\int_{d_1^2/2}^{2\sqrt{p}}\frac{y}{d_1^2}\sum_{\substack{q\ge (2y/d_1^2)^\delta\\ P^+(q)\le z\\ \rad(q),\frac{q}{\rad(q)}\ll p}} \frac{c^{\omega(q)}\max_{a\in\Z/q}|C_{h,p}(a,d_1,q)|}{(q\rad(q))^{1/2}(q,d_1^2)^{1/2}}d_py\\
  &\ll&\sum_{\substack{d_1\mid p-1\\ d_1\ll p^{1/4}}} \frac{1}{d_1^{2}}\frac{1}{p}\int_{d_1^2/2}^{2\sqrt{p}}y\\
  &&\hspace{1cm}\sum_{\substack{n\ge 1\\ P^+(n)\le z}}\frac{\mu(n)^2c^{\omega(n)}}{n}\sum_{\substack{q'n\gg (y/d_1^2)^\delta\\ \rad(q')\mid n}} \frac{\max_{a\in\Z/q'n}|C_{h,p}(a,d_1,q'n)|}{(q')^{1/2}(qn',d_1^2)^{1/2}}d_py,
\end{eqnarray*}
where we let $q'=q/\rad(q)$. Since either $q'\gg (y/d_1^2)^{\delta/2}$ or $n\gg (y/d_1^2)^{\delta/2}$, this is
\begin{eqnarray*}
&\ll&\sum_{i=1}^2F_{h,p}\left(z,\frac{\beta_i\delta}{4},\frac{1}{4},2-\beta_i\delta,1-\delta_{i=2}\beta_2,\frac{1}{2}-\delta_{i=1}\beta_1\right),
\end{eqnarray*}
using Lemma \ref{lemma:int}, giving the first part of $E_{h,p}^{(L2)}$.

The contribution of the range \ref{item:PPartsSI} above is
\begin{eqnarray*}
  &&\sum_{\substack{d_1\mid p-1\\ d_1\ll p^{1/2}}}\sum_{\substack{q\ge 1\\ P^+(q)\le z}}\frac{c^{\omega(q)}}{\rad(q)}\frac{1}{p}\int_0^{\min(d_1^2/2,2\sqrt{p})} \frac{y}{d_1^2}\max_{\frac{p_-}{d_1^2}< d_2\le \frac{p_+}{d_1^2}}|h(d_1,d_2)|d_py\\
  &&+\sum_{\substack{d_1\mid p-1\\ d_1\ll p^{1/2}}}\sum_{\substack{q\ge 1\\ P^+(q)\le z}}\frac{c^{\omega(q)}\max_{a\in\Z/q}|C_{h,p}(a,d_1,q)|}{(q\rad(q))^{1/2}(q,d_1^2)^{1/2}}\frac{1}{p}\int_0^{\min(d_1^2/2,2\sqrt{p})} \frac{y}{d_1^2}d_py.
\end{eqnarray*}
By Lemma \ref{lemma:int}, the first summand is
  \begin{eqnarray*}
    &\ll&\frac{(\log{z})^{O(1)}}{p^{\frac{1-\beta_4-\varepsilon}{2}}}\sum_{\substack{d_1\mid p-1\\ d_1\ll p^{1/2}}}\frac{1}{d_1^{2\beta_4}}\max_{\frac{p_-}{d_1^2}<d_2\le \frac{p_+}{d_1^2}}|h(d_1,d_2)|,
  \end{eqnarray*}
  which gives the remaining of $E_{h,p}^{(L1)}$. Similarly, the second summand is
  \begin{eqnarray*}
    &\ll&\frac{p^\varepsilon}{p^{\frac{1-\beta_3}{2}}}\sum_{\substack{d_1\mid p-1\\ d_1\ll p^{1/2}}}\frac{1}{d_1^{2\beta_3}}\sum_{\substack{n\ge 1\\ P^+(n)\le z}}\frac{\mu(n)^2c^{\omega(n)}}{n}\sum_{\substack{q'\ge 1\\\rad(q')\mid n}}\frac{\max_{a\in\Z/q'n}|C_{h,p}(a,d_1,q'n)|}{(q')^{1/2}(q'n,d_1^2)^{1/2}}\\
    &\ll&p^\varepsilon F_{h,p}\left(z,\frac{1-\beta_3}{2},\frac{1}{2},2\beta_3,1,\frac{1}{2}\right),
  \end{eqnarray*}
  which gives the second part of $E_{h,p}^{(L2)}$.
\end{proof}
To estimate $F_{h,p}$, we see directly from the decomposition \eqref{eq:hd2modu} of $C_{h,p}$ that:
  \begin{lemma}\label{lemma:Fhp}
  If \eqref{eq:hd2modu} holds, then
  \begin{eqnarray*}
    F_{h,p}(z,\bs\nu)&\ll&\frac{1}{p^{\nu_1}}\sum_{\substack{d_1\mid p-1\\ \bs v\in\N^m\\d_1\ll p^{\nu_2}}}\frac{|C^{(1)}_{h,p}(\bs v,d_1)|}{d_1^{\nu_3}}\sum_{\substack{n\ge 1\\ P^+(n)\le z}}\frac{\mu(n)^2O(1)^{\omega(n)}}{n^{\nu_4}}\\
                        &&\sum_{\substack{q\ge 1\\\rad(q)\mid n}}\frac{\max_{a\in\Z/qn}|C^{(2)}_{h,p}(\bs v,d_1,qn)|}{q^{\nu_5}(qn,d_1^2)^{1/2}}\prod_{\ell\nmid q}|C^{(3)}(\bs v,d_1,\ell)|.
  \end{eqnarray*}
\end{lemma}

\section{The number of subgroups ($h=s$)}\label{sec:hs}

In this section, we finally prove the first part of Theorem \ref{thm:scEll}. It remains to check that the hypotheses of Theorem \ref{thm:general} hold and to bound the errors. Again, all implied constants may depend on a parameter $\varepsilon>0$.

\subsection{Number of subgroups of an abelian group of rank at most 2}

The starting point is the following expression for $h=s$, that we already mentioned in \eqref{eq:sExplicit}:
\begin{proposition}\label{prop:scRank2} For all integers $d_1,d_2\ge 1$,
    \begin{eqnarray*}
    s(\Z/d_1\times\Z/d_1d_2)&=&\sum_{u\mid d_1}\varphi(u)\tau(d_1/u)\tau(d_1^2d_2/u).
  \end{eqnarray*}
\end{proposition}
\begin{proof}
  By \cite{Calh87}, the number of subgroups of $\Z/m\times\Z/n$ is
      \begin{eqnarray*}
        \sum_{\substack{a\mid m\\ b\mid n}} \left(a,\frac{(m+1)^n-1}{(m+1)^b-1}\right)&=&\sum_{\substack{a\mid m\\ b\mid n}} \left(a,1+(m+1)+\dots+(m+1)^{b-1}\right)\\
                                                                                      &=&\sum_{\substack{a\mid m, \, b\mid n}} \left(a,b\right)=\sum_{\substack{a\mid m, \, b\mid n}} \sum_{u\mid (a,b)}\varphi(u)\\
                                                                                      &=&\sum_{u\mid (m,n)}\varphi(u)\tau(m/n)\tau(n/u).
      \end{eqnarray*}
      Alternatively, see \cite[Theorem 4]{HampHolTothWies14}.
\end{proof}
\subsection{The densities $w_{s,p}$ in arithmetic progressions}
To apply Theorem \ref{thm:general}, we start by proving that \eqref{eq:hd2modu} holds.
\begin{proposition}\label{prop:Chs}
  In the case $h=s$, Equation \eqref{eq:hd2modu} holds with $m=3$, $\bs v=(u,k,i)\in\N^3$,
  \begin{eqnarray*}
    C^{(1)}_{s,p}(\bs v,d_1)&=&\delta_{u\mid d_1}\varphi(u)\tau(d_1/u)\left(\log \left(\frac{p+1}{uk^2}\right)+2\gamma\right)\frac{\delta_{i=0}\delta_{k\mid d_1^2/u}\varphi(k)+\delta_{i=1}}{k},\\
  \end{eqnarray*}
  \begin{eqnarray*}
    C^{(2)}_{s,p}(a,\bs v,d_1,q)&=&(d_1^2,q)\left(\delta_{i=0}+\delta_{i=1}c_{(k,q)} \left(\frac{a}{(d_1^2,q)}\right)\prod_{\ell\mid q}\delta_{v_\ell(q)\ge v_\ell(d_1^2k)}\right),\\
    C^{(3)}_{s,p}(\bs v,d_1,\ell)&=&\delta_{i=0}+\delta_{i=1}\delta_{\ell\nmid k}.
  \end{eqnarray*}
  Moreover, if $1\le q\le 2y/d_1^2$, then, under the notations of Section \ref{sec:DeltaqAB},
  \begin{eqnarray*}
    |E'_{s,p}(y,\bs v,d_1,a,q)|&\ll&\left|\Delta \left(\frac{p+1\mp y}{u},A_{a,u,q},\frac{qd_1^2}{u(d_1^2,q)}\right)\right|+p^\varepsilon \frac{y^2(d_1^2,q)}{pqd_1^2},
  \end{eqnarray*}
  where $A_{a,u,q}\equiv 0\pmod{d_1^2/u}$ and $uA_{a,u,q}\equiv a\pmod{q}$.
  \end{proposition}
  \begin{remark}
    The variable $u\mid d_1$ in $\bs v$ comes from the sum in the explicit expression for $s$ (Proposition \ref{prop:scRank2}), $k\mid d_1^2/u$ comes from the sum in the main term for $\tau$ in arithmetic progressions, and $i\in\{0,1\}$ comes from two different cases in the evaluation of the Ramanujan sums.
  \end{remark}
\begin{proof}
  We need to compute the left-hand side of \eqref{eq:hd2mod}, which is, by Proposition \ref{prop:Chs},
  \begin{eqnarray*}
  \sum_{\substack{\frac{p+1-y}{d_1^2}<d_2\le \frac{p+1+y}{d_1^2}\\ d_1^2d_2\equiv a\pmod{q}}}s(d_1,d_2)&=&\sum_{u\mid d_1}\varphi(u)\tau(d_1/u)\sum_{\substack{\frac{p+1-y}{d_1^2}<d_2\le \frac{p+1+y}{d_1^2}\\ d_1^2d_2\equiv a\pmod{q}}}\tau(d_2\cdot d_1^2/u).
  \end{eqnarray*}
The inner sum can be rewritten as
\begin{eqnarray*}
  \sum_{\substack{\frac{p+1-y}{u}<d_2'\le \frac{p+1+y}{u}\\ ud_2'\equiv a\pmod{q}\\ d_2'\equiv 0\pmod{d_1^2/u}}}\tau(d_2')&=&\delta_{(d_1^2,q)\mid a}\sum_{\substack{\frac{p+1-y}{u}<d_2'\le \frac{p+1+y}{u}\\ \frac{u}{(u,q)}d_2'\equiv \frac{a}{(u,q)}\pmod{\frac{q}{(u,q)}}\\ d_2'\equiv 0\pmod{d_1^2/u}}}\tau(d_2').
\end{eqnarray*}
Note that $[q/(u,q),d_1^2/u]=\frac{qd_1^2}{(d_1^2,q)u}$ and if $1\le A<B$ are such that $\frac{B-A}{2A}<1$, then
\[[x\log{x}]_A^B=(B-A)\left(\log \left(\frac{A+B}{2}\right)+1 + O \left(\frac{B-A}{A}\right)\right).\]

If $(d_1^2,q)\mid a$, then the left-hand side of \eqref{eq:hd2mod} is given by
\begin{eqnarray*}
  &&\frac{2y}{qd_1^2}\sum_{u\mid d_1}\varphi(u)\tau(d_1/u)(d_1^2,q)\\
  &&\hspace{2cm}\sum_{k\mid [q/(u,q),d_1^2/u]}\frac{c_k(A_{a,u,q})}{k}\left(\log \left(\frac{p+1}{uk^2}\right)+2\gamma+ O \left(\frac{y}{p}\right)\right)\\
  &&\hspace{1cm}+O \left(\sum_{u\mid d_1}\varphi(u)\tau(d_1/u)\left|\Delta \left(\frac{p+1\mp y}{u},A_{a,u,q},\frac{qd_1^2}{u(d_1^2,q)}\right)\right|\right).
\end{eqnarray*}

The main term is
\begin{eqnarray*}
  &&\frac{2y}{qd_1^2}\sum_{u\mid d_1}\varphi(u)\tau(d_1/u)(d_1^2,q)\sum_{k\mid [q/(u,q),d_1^2/u]}\frac{c_k(A_{a,u,q})}{k}\left(\log \left(\frac{p+1}{uk^2}\right)+2\gamma\right)\\
  &&\hspace{2cm}+ O \left(\frac{y^2}{pqd_1^2}\sum_{u\mid d_1}\varphi(u)\tau(d_1/u)(d_1^2,q)\tau(qd_1^2)\right).
\end{eqnarray*}
We write the sum over $k$ as
\[\left[\sum_{k\mid d_1^2/u} \varphi(k)+ \sum_{k\mid q/(d_1^2,q)} c_k(a/(d_1^2,q))\right]\frac{1}{k}\left(\log \left(\frac{p+1}{uk^2}\right)+2\gamma\right)(d_1^2,q).\]
and the two summands correspond respectively to the cases $i=0$ and $i=1$ in the statement. Finally, note that for a fixed integer $a$, $q\mapsto c_q(a)$ is multiplicative.
\end{proof}

\subsection{Main term}\label{subsec:mainTerm}
\begin{proposition}\label{prop:maintermS}
  Hypotheses \eqref{eq:stabilizes}--\eqref{eq:dependsvl} hold with
  \[r_{\ell,\bs v,d_1}=v_\ell(d_1^2)+\delta_{i=1}\max(0,v_\ell(k)-1).\]
\end{proposition}
\begin{proof}
  The claim is clear for $i=0$. For $i=1$, it follows from von Sterneck's formula \cite[(3.3)]{IK04} that the Ramanujan sum is given by
  \begin{equation}
    \label{eq:vonSterneck}
   c_{\ell^r}(a)=\ell^r
    \begin{cases}
      0&:v_\ell(a)<r-1\\
      (-1)/\ell&:v_\ell(a)=r-1\\
      1-1/\ell&:v_\ell(a)\ge r
    \end{cases}
    \hspace{1cm}(r\ge 1). 
  \end{equation}
\end{proof}

  \begin{lemma}\label{lemma:computationEPPractice}
    For $C_{s,p}^{(2)}$ and $C^{(3)}_{s,p}$ as in Proposition \ref{prop:Chs}, the Euler product in Theorem \ref{thm:general} is
    \begin{eqnarray*}
      &&\prod_{\ell}\left[C_{s,p}^{(3)}(\bs v,d_1,\ell)+L_{s,p}(\bs v,d_1,\ell,r_{\ell,\bs v,d_1})\right]\\
      &=&\prod_{\substack{\ell\nmid d_1\\ \ell\mid p-1}}\left(1-\frac{1}{\ell(\ell^2-1)}\right)
          \begin{cases}
            \prod_{\substack{\ell\mid d_1\\ \ell\nmid k}}\ell^{-v_\ell(d_1)}\oO{1/\ell}&:i=0\\
            \delta_{k=1}\prod_{\substack{\ell\mid d_1}}\ell^{-v_\ell(d_1)}\oO{1/\ell}&:i=1.
          \end{cases}                                                                                                        
    \end{eqnarray*}
  \end{lemma}
  \begin{proof}
    When $i=0$,
    \begin{eqnarray*}
      L'_{s,p}(\bs v,d_1,\ell,r_{\ell,\bs v,d_1})&=&\left(1-\frac{1}{\ell}\right)\sum_{w\ge v_\ell(d_1^2)}\frac{\ell^{v_\ell(d_1^2)}}{\ell^w}=1,\text{ so}\\
      L_{s,p}(\bs v,d_1,\ell,r_{\ell,\bs v,d_1})&=&\frac{1}{\ell^{v_\ell(d_1)}}\oO{\frac{1}{\ell}}-1
    \end{eqnarray*}
    (see Definition \ref{def:L'}), while when $i=1$, \eqref{eq:vonSterneck} gives
  \begin{eqnarray*}
    L'_{s,p}(\bs v,d_1,\ell,r_{\ell,\bs v,d_1})&=&\ell^{v_\ell(d_1^2k)}\left(-\frac{\delta_{v_\ell(k)\ge 1}}{\ell^{v_\ell(d_1^2k)}}+\frac{1}{\ell^{v_\ell(d_1^2k)}}\right)=\delta_{\ell\nmid k},\\
    L_{s,p}(\bs v,d_1,\ell,r_{\ell,\bs v,d_1})&=&\delta_{\ell\nmid k}\left(\frac{1}{\ell^{v_\ell(d_1)}}\oO{\frac{1}{\ell}}-1\right).
  \end{eqnarray*}
  If $\ell\nmid d_1kp$, the last part of Proposition \ref{prop:localDensities} gives
  \begin{eqnarray*}
    L_{s,p}(\bs v,d_1,\ell,r_{\ell,\bs v,d_1})&=&\lim_{R\to\infty}\sum_{w=0}^R C^{(2)}(\ell^w,\bs v,d_1,\ell^R)g_p(w,v_\ell(d_1),\ell^R)\\
                           &=&\lim_{R\to\infty}\sum_{w=0}^R g_p(w,v_\ell(d_1),\ell^R)=\frac{-\delta_{\ell\mid p-1}}{\ell(\ell^2-1)}.
  \end{eqnarray*}
  \end{proof}
\subsection{Estimation of the error $E_{s,p}^{(P)}$}
We use Proposition \ref{prop:splitEM} to estimate the former from $E_{s,p}^{(L1)}$, $E_{s,p}^{(L2)}$ and $E_{s,p}^{(S)}$. Accordingly, let $\varepsilon>0$, $\delta\in(0,1]$ and $\bs\beta\in[0,1/2]\times[0,1]^3$.\\

We make the following definition, so that $\prod_{\ell\le z}\left(1+\ell^{-\nu}\right)\ll f(z,\nu)^{O(1)}$:
\begin{definition}
  For $z\ge 1$ and $\nu\in[0,1]$, we let
  \[f(z,\nu)=\begin{cases}
      \log{z}&:\nu=1\\
      \exp \left(\frac{z^{1-\nu}}{\log{z}}\right)&:\nu<1.
    \end{cases}\]  
\end{definition}

We will show:
\begin{proposition}\label{prop:EPs}
  If $\beta_2\le 1/(2+\delta)$ and $\delta<1/2$, then
\begin{eqnarray*}
  E_{s,p}^{(P)}(z)&\ll&p^\varepsilon \left(\frac{ f(z,1-\beta_2)^{O(1)}}{p^{\beta_2\delta/4}}+\frac{(\log{z})^{O(1)}}{p^{(1-2\delta)/8}}\right).
\end{eqnarray*}
\end{proposition}
   \begin{remark}\label{rem:choiceGamma}
    At the end, we will choose $z=(\log{p})^A$ with $A\ge 1$ large. If we want $f(z,\nu)=O(p^\varepsilon)$ and $\nu<1$, we must pick $\nu\ge 1-1/A$.
   \end{remark}
\subsubsection{Estimation of $E_{s,p}^{(S)}$}
\begin{lemma}
  $E_{s,p}^{(S)}(z,\delta)\ll\frac{(\log{z})^{O(1)}}{p^{(1-2\delta)/8-\varepsilon}}$.
\end{lemma}
\begin{proof}
  By definition,
  \[E_{s,p}^{(S)}(z,\delta)\ll\sum_{\substack{u\mid d_1\mid p-1\\ d_1\ll p^{1/4}}}ud_1^\varepsilon\frac{1}{p}\int_{d_1^2/2}^{2\sqrt{p}}\sum_{\substack{q\le (2y/d_1^2)^\delta\\ q\in Q(d_1,z)}}\frac{c^{\omega(q)}}{\rad(q)} \sum_{\substack{a\in\Hc(q)\\(d_1^2,q)\mid a}}|E_{s,p}(y,d_1,a,q)|d_py\]
  for some $c\ge 1$. Let $Q=\frac{qd_1^2}{u(d_1^2,q)}\le \frac{2y}{u(d_1^2,q)}$. By Proposition \ref{prop:Chs} and Theorem \ref{thm:tauAPSI}, the sum over $a$ is
  \begin{eqnarray*}
    &\ll&\sum_{\substack{a\in\Z/\frac{q}{(d_1^2,q)}}}\left|\Delta\left(\frac{p+1\mp y}{u},A_{a,u,q},Q\right)\right|+p^\varepsilon\frac{q}{(d_1^2,q)}\frac{y^2(d_1^2,q)}{pqd_1^2}\\
    &\ll&\frac{u}{d_1^2}\sum_{\substack{a\in\Z/Q}}\left|\Delta\left(\frac{p+1\mp y}{u},a,Q\right)\right|+p^\varepsilon\frac{y^2}{pd_1^2}\\
    &\ll&\frac{uQ^{1/2}}{d_1^2}\left(\sum_{a\in\Z/Q}\left|\Delta\left(\frac{p+1\mp y}{u},a,Q\right)\right|^2\right)^{1/2}+p^\varepsilon\frac{y^2}{pd_1^2}\\
    &\ll&\frac{q^{1/2}}{d_1(d_1^2,q)^{1/2}}(py)^{1/4}+p^\varepsilon\frac{y^2}{pd_1^2},
  \end{eqnarray*}
  where $A_{a,u,q}\equiv 0\pmod{d_1^2/u}$, $u A_{a,u,q}\equiv a\pmod {q}$, $\frac{u}{(u,q)} A_{a,u,q}\equiv \frac{a}{(u,q)}\pmod{\frac{q}{(u,q)}}$. Hence $E_{s,p}^{(S)}(z,\delta)$ is
    \begin{eqnarray*}
    &\ll&p^{1/4+\varepsilon}\sum_{\substack{u\mid d_1\mid p-1\\ d_1\ll p^{1/4}}}\frac{u}{d_1}\frac{1}{p}\int_{d_1^2/2}^{2\sqrt{p}} y^{1/4}\sum_{\substack{q\le (2y/d_1^2)^\delta\\ q\in Q(d_1,z)}}\frac{c^{\omega(q)}}{\rad(q)} q^{1/2}d_py\\
                    &&+\frac{1}{p^{1-\varepsilon}}\sum_{u\mid d_1\mid p-1}\frac{u}{d_1^2}\frac{1}{p}\int_0^{2\sqrt{p}}y^2\sum_{\substack{q\le (2y/d_1^2)^\delta \\ q\in Q(d_1,z)}} \frac{c^{\omega(q)}}{\rad(q)}d_py\\
                    &\ll&(\log{z})^{O(1)}\left(p^{1/4+\varepsilon}\sum_{\substack{d_1\mid p-1\\ d_1\ll p^{1/4}}}d_1^{-\delta}\frac{1}{p}\int_{d_1^2/2}^{2\sqrt{p}} y^{1/4+\delta/2}d_py+\frac{1}{p^{1/2-\varepsilon}}\right)\ll\frac{(\log{z})^{O(1)}}{p^{\frac{1-2\delta}{8}-\varepsilon}}.
  \end{eqnarray*}
\end{proof}
\begin{remark}
  At this point, one can check that applying only \eqref{eq:Voronoi1} instead of Theorem \ref{thm:tauAPSI} at best yields the non-admissible bound $(\log{z})^{O(1)}p^\varepsilon$, if $\delta<1/3$.
\end{remark}

\subsubsection{Bound on $E_{s,p}^{(L1)}$}
By Remark \ref{rem:L1h}, we have $E_{s,p}^{(L1)}(1/2)\ll 1/p^{1/4-\varepsilon}$, since $s(d_1,d_2)\ll d_1(d_1d_2)^\varepsilon$ for any $\varepsilon>0$.
\subsubsection{Bound on $E_{s,p}^{(L2)}$}
\begin{lemma}\label{lemma:EL2s}
  If $\beta_2<1/(2+\delta)$, then
  \[E_{s,p}^{(L2)}(z,(\beta_2,\beta_2,1/2),\delta)\ll f(z,1-\beta_2)^{O(1)}p^{-\beta_2\delta/4+\varepsilon}.\]
\end{lemma}
We start with a preliminary lemma bounding $F_{s,p}$ (cf. Definition \ref{def:EPaux}):
\begin{lemma}\label{lemma:FC2d1uq}
  With $C_{s,p}^{(1)},C_{s,p}^{(2)}$ as in Proposition \ref{prop:Chs}, for any $\bs\nu\in [0,1]^4\times (0,1/2)$, we have
  \begin{eqnarray*}
    F_{s,p}(z,\bs\nu)&\ll&\frac{f(z,\nu_4)^{O(1)}}{p^{\nu_1-\varepsilon}}\sum_{\substack{d_1\mid p-1\\d_1\ll p^{\nu_2}}} d_1^{3-\nu_3-\nu_4-2\nu_5}.
  \end{eqnarray*}
\end{lemma}

   \begin{proof}
     By Lemma \ref{lemma:Fhp}, $F_{s,p}(z,\bs\nu)$ is
     \begin{eqnarray*}
       &\ll&\frac{1}{p^{\nu_1-\varepsilon}}\sum_{\substack{u\mid d_1\mid p-1\\d_1\ll p^{\nu_2}}}\frac{u}{d_1^{\nu_3}}\sum_{\substack{n\ge 1\\ P^+(n)\le z}}\frac{\mu(n)^2c^{\omega(n)}}{n^{\nu_4}}\sum_{\substack{q\ge 1\\\rad(q)\mid n}}\frac{(qn,d_1^2)}{q^{\nu_5}(qn,d_1^2)^{1/2}}\\
       &\ll&\frac{1}{p^{\nu_1-\varepsilon}}\sum_{\substack{u\mid d_1\mid p-1\\d_1\ll p^{\nu_2}}}\frac{u}{d_1^{\nu_3}}\sum_{\substack{n\ge 1\\ P^+(n)\le z}}\frac{\mu(n)^2c^{\omega(n)}(n,d_1^2)^{1/2}}{n^{\nu_4}}\sum_{\substack{q\ge 1\\\rad(q)\mid n}}\frac{(q,d_1^2)^{1/2}}{q^{\nu_5}}.
  \end{eqnarray*}
  \end{proof}
  \begin{proof}[Proof of Lemma \ref{lemma:EL2s}]    
    By Lemma \ref{lemma:FC2d1uq}, $E_{s,p}^{(L2)}(z,\bs\beta,\delta)$ is
         \begin{eqnarray*}
       &\ll&p^\varepsilon\left[\begin{cases}
         p^{-\beta_1\delta/4}&\beta_1\in[0,\frac{1}{2+\delta}]\\
         p^{-(1-2\beta_1)/4}&\beta_1\in[\frac{1}{2+\delta},\frac{1}{2}]
       \end{cases}
+
                              \begin{cases}
                                p^{-(1-\beta_3)/4}&\beta_3\in[0,1/2]\\
                                p^{-(3-5\beta)/4}&\beta_3\in[1/2,3/5]
                              \end{cases}\right.
\\
                                 &&\left.\hspace{3cm}+f(z,1-\beta_2)^{O(1)}
\begin{cases}
  p^{-\beta_2\delta/4}&\beta_2\in[0,\frac{1}{1+\delta}]\\
  p^{-(1-\beta_2)/4}&\beta_2\in[\frac{1}{1+\delta},1]
\end{cases}\right].
     \end{eqnarray*}
   \end{proof}
\subsubsection{Conclusion}
By Proposition \ref{prop:splitEM} and the above, if $\beta_2\le 1/(2+\delta)$ and $\bs\beta=(\beta_2,\beta_2,1/2,1/2)$, then
\begin{eqnarray*}
  E_{s,p}^{(P)}(z)&\ll&(\log{z})^{O(1)}p^\varepsilon E_{s,p}^{(L1)}(\beta_4)+p^\varepsilon E_{s,p}^{(L2)}(z,\bs\beta,\delta)+E_{s,p}^{(S)}(z,\delta)\\
                                                  &\ll&p^\varepsilon \left(\frac{(\log{z})^{O(1)}}{p^{1/4}}+\frac{ f(z,1-\beta_2)^{O(1)}}{p^{\beta_2\delta/4}}+\frac{(\log{z})^{O(1)}}{p^{(1-2\delta)/8}}\right)\\
                                                  &\ll&p^\varepsilon \left(\frac{ f(z,1-\beta_2)^{O(1)}}{p^{\beta_2\delta/4}}+\frac{(\log{z})^{O(1)}}{p^{(1-2\delta)/8}}\right).
\end{eqnarray*}
This proves Proposition \ref{prop:EPs}.\qed
\subsection{Estimation of the remaining error terms}
\subsubsection{Bound on $E_{s,p}^{(G)}$}
\begin{lemma}\label{lemma:EsG}
  $E_{s,p}^{(G)}\ll (\log{p})^5\log_2{p}$.
\end{lemma}
\begin{proof}
  By Lemma \ref{lemma:EGHyp2} and Proposition \ref{prop:Chs}, we have
\begin{eqnarray*}
  E_{s,p}^{(G)}&\ll&\log{p}\sum_{u\mid d_1\mid p-1} \frac{\varphi(u)\tau(d_1/u)\tau(d_1^2/u)}{d_1^2}\\
               &&+\sum_{u\mid d_1\mid p-1}\frac{u}{p^{1-\varepsilon}}\int_0^{2\sqrt{p}}\left(\left|\Delta\left(\frac{p+1\mp y}{u},0,\frac{d_1^2}{u}\right)\right|+\frac{yu}{pd_1^2}\right)d_py.
\end{eqnarray*}

The first summand is
\begin{eqnarray*}
  &\ll&\log{p}\sum_{d_1\mid p-1} \frac{1}{d_1^2}\sum_{u\mid d_1}\varphi(u)\tau(d_1/u)\tau(d_1^2/u)\\
  &\ll&\log{p}\sum_{d_1\mid p-1} \frac{\tau(d_1^2)}{d_1^2}\sigma(d_1)\ll\log{p}\log_2{p}\sum_{d_1\le p-1} \frac{\tau(d_1)^2}{d_1}\ll(\log{p})^5\log_2{p},
\end{eqnarray*}
since $\sum_{d\le n} \tau(d)^2=\Theta(n(\log{n})^3)$.

The second summand is
\begin{eqnarray*}
  &\ll&\sum_{u\mid d_1\mid p-1}\frac{u}{p^{1-\varepsilon}}\int_0^{2\sqrt{p}}\left(\left|\Delta\left(\frac{p+1\mp y}{u},0,\frac{d_1^2}{u}\right)\right|\right)d_py+\frac{1}{p^{1-\varepsilon}}.
\end{eqnarray*}
Let $\theta\in(0,1)$. By Theorem \ref{thm:tauAPSI}, the contribution of $y\ge (d_1^2/u)^{1/\theta}$ and that of $y\le (d_1^2/u)^{1/\theta}\le 2\sqrt{p}$ is
\begin{eqnarray*}
  &\ll&p^\varepsilon\sum_{\substack{u\mid d_1\mid p-1\\ d_1\ll p^{\theta/4}}}u\frac{1}{p}\int_{0}^{2\sqrt{p}} \frac{p^{1/4}y^{1/4}}{u^{1/2}}d_py\ll p^{-(1-\theta)/8+\varepsilon}.
\end{eqnarray*}
The remaining error is
\begin{eqnarray*}
  &\ll&\frac{1}{p^{1-\varepsilon}}+\sum_{\substack{u\mid d_1\mid p-1 \\ d_1^2/u\gg p^{\theta/2}}}\frac{u}{p^{1-\varepsilon}}\int_0^{2\sqrt{p}}\left(\frac{y}{u}\frac{u}{d_1^2}+1\right)d_py\ll p^{-1/2+\theta/4+\varepsilon}.
\end{eqnarray*}
\end{proof}

\subsubsection{Bound on $E_{s,p}^{(B)}$} By Remark \ref{rem:Bh}, we have $E_{s,p}^{(B)}\ll 1/p^{1/2-\varepsilon}$.
\subsubsection{Bound on $E_{s,p}^{(T)}$} By Section \ref{subsec:mainTerm}, we have
\begin{eqnarray*}
  E_{s,p}^{(T)}(z,\bs v,d_1)&\ll&\sum_{\substack{\ell>z\\ \ell\mid d_1}}\frac{1}{\ell^{v_\ell(d_1)}}+\sum_{\substack{\ell>z\\ \ell\nmid d_1\\ \ell\mid p-1}}\frac{1}{\ell^3}\ll\frac{\omega(d_1)}{z}\ll\frac{\log{p}}{z\log_2{p}}.
\end{eqnarray*}
\subsection{Conclusion}
By Theorem \ref{thm:general}, Lemma \ref{lemma:computationEPPractice}, Proposition \ref{prop:Chs}, and the estimations above, we get that $s(\Ell(p))$ is equal to
  \begin{eqnarray*}
    &&\sum_{u\mid d_1\mid p-1} \frac{\varphi(u)\tau(d_1/u)}{d_1^3}\sum_{k\mid d_1^2/u}\left(\log \left(\frac{p+1}{uk^2}\right)+2\gamma\right)\\
    &&\hspace{1cm}\frac{\varphi(k)+\delta_{k=1}}{k}\prod_{\ell\mid k}\ell^{v_\ell(d_1)}\oO{\frac{1}{\ell}}\prod_{\substack{\ell\nmid d_1\\ \ell\mid p-1}}\left(1-\frac{1}{\ell(\ell^2-1)}\right)\\
    &&+O \Bigg(\frac{Z^{\frac{2}{\alpha}+\varepsilon}}{p^{1/2-\varepsilon}}+\frac{(\log{p})^{5+\varepsilon}}{(\log{Z})^{\alpha-\varepsilon}}+\frac{ f(z,1-\beta_2)^{O(1)}}{p^{\beta_2\delta/4-\varepsilon}}+\frac{(\log{z})^{O(1)}}{p^{(1-2\delta)/8-\varepsilon}}+\frac{(\log{p})^6}{z}\Bigg),
  \end{eqnarray*}
  for any $\alpha\ge 1$, $\delta\in(0,1/2)$, $\beta_2\in[0,1/(2+\delta)]$, $Z\ge \exp(\sqrt{\log(4p)})$, and $z=(\log{Z})^{8\alpha^2}$.

  We choose $Z=p$ and $\beta_2\le 1/(8\alpha^2)$, so that $f(z,1-\beta_2)=O(p^\varepsilon)$ according to Remark \ref{rem:choiceGamma}. By taking $\alpha$ large enough, the error can be made $\ll_A 1/(\log{p})^A$ for any $A>0$. This proves the first part of Theorem \ref{thm:scEll}.\qed

  \begin{remark}
    If there are no Siegel zeros, according to Proposition \ref{prop:truncateComplete}, the first summand of the error above can be removed, and the limiting factor becomes the third summand.
  \end{remark}
  
  \subsection{Order of magnitude}\label{sec:orderMags}
  
\subsubsection{Average over primes}

  \begin{proof}[Proof of Proposition \ref{prop:averagep}]
    Let $\rho$ be the totally multiplicative function defined by $\rho(\ell)=-(\ell^3-\ell-1)$ for every prime $\ell$. By Abel's summation formula, we have
    \begin{eqnarray*}
      &&\frac{1}{\pi(x)}\sum_{p\le x} s(\Ell(p))=\sum_{m\le x}\frac{1}{\rho(m)}\sum_{\substack{d_1\le x\\ u\mid d_1}} \frac{\varphi(u)\tau(d_1/u)}{d_1^3}\sum_{k\mid d_1^2/u}\frac{\varphi(k)+\delta_{k=1}}{k}\\
     &&\hspace{1cm}\prod_{\ell\mid k}\ell^{v_\ell(d_1)}\oO{\frac{1}{\ell}}\prod_{\substack{\ell\mid d_1}}\left(1-\frac{1}{\ell(\ell^2-1)}\right)^{-1}\\
      &&\hspace{1cm}\left[\frac{\pi(x,1,[d_1,m])}{\pi(x)}\left(\log \left(\frac{x+1}{uk^2}\right)+2\gamma\right)\right.\\
      &&\hspace{4cm}\left.+ O \left(\frac{1}{\pi(x)}\int_2^x \frac{\pi(t,1,[d_1,m])}{t}dt\right)\right]+O \left(1\right),
    \end{eqnarray*}
    where $\pi(x,a,q)=|\{p\le x: p\equiv a\pmod{q}\}|$ as usual.

   By the Siegel--Walfisz theorem \cite[Corollary 5.29]{IK04}, the expression between square brackets is
   \begin{eqnarray*}
     &&\left(\frac{1}{\varphi([d_1,m])}+ O_A \left(\frac{1}{(\log{x})^{A-1}}\right)\right)\left(\log \left(\frac{x+1}{uk^2}\right)+2\gamma\right)+ O \left(1\right)\\
     &=&\frac{\log(x+1)}{\varphi([d_1,m])}+O_A(1),
   \end{eqnarray*}
   for any constant $A>1$. So the average above is
      \begin{eqnarray*}
        &&\log(x+1)\sum_{m\le x}\frac{1}{\rho(m)}\sum_{\substack{d_1\le x\\ u\mid d_1}} \frac{\varphi(u)\tau(d_1/u)}{d_1^3\varphi([d_1,m])}\sum_{k\mid d_1^2/u}\frac{\varphi(k)+\delta_{k=1}}{k}\\
     &&\hspace{1cm}\prod_{\ell\mid k}\ell^{v_\ell(d_1)}\oO{\frac{1}{\ell}}\prod_{\ell\mid d_1}\left(1-\frac{1}{\ell(\ell^2-1)}\right)^{-1}+O(1)\\
     &=&(C_s+o(1))\log(x+1),
      \end{eqnarray*}
      for some constant $C_s\ge 1$.
    \end{proof}
    
\subsubsection{Upper and lower bounds}
\begin{proof}[Proof of Proposition \ref{prop:upperlowerBounds}]
  For the upper bound, we have
   \begin{eqnarray*}
     s(\Ell(p))&\ll&\log{p}\sum_{u\mid d_1\mid p-1} \frac{\varphi(u)\tau(d_1/u)}{d_1^2}\sum_{k\mid d_1^2/u}\frac{\varphi(k)}{k}\prod_{\ell\mid k}\oO{\frac{1}{\ell}}.
   \end{eqnarray*}
   Going through Section \ref{sec:proofs} with \eqref{eq:fellPrecise}, we see that the product can actually be replaced by
   \[\prod_{\ell\mid k} \left(1+\frac{1}{\ell}\left(1+\frac{2}{\ell-1}\right)\right).\]
   This gives
   \begin{eqnarray*}
     s(\Ell(p))&\ll&\log{p}\sum_{d_1\mid p-1} \frac{1}{d_1^2} \sum_{u\mid d_1} \varphi(u)\tau(d_1/u)\sum_{k\mid d_1^2/u}\exp \left(\frac{\sigma(k)}{k}\right)\\
               &\ll&(\log{p})^{1+e^\gamma+\varepsilon}\sum_{d_1\mid p-1} \frac{\tau(d_1^2)}{d_1^2}\sigma(d_1)\ll(\log{p})^{1+e^\gamma+\varepsilon}\log_2{p}\sum_{d_1\mid p-1} \frac{\tau(d_1^2)}{d_1}\\
               &\ll&(\log{p})^{1+e^\gamma+\varepsilon}\log_2{p}\min \left((\log{p})^4, \tau((p-1)^2) \frac{\sigma(p-1)}{p-1}\right)
   \end{eqnarray*}
   (see also Lemma \ref{lemma:EsG}).

   For the lower bound, note that if $k\mid d_1^2/u$, then
    \[\frac{p+1}{uk^2}\ge \frac{(p+1)u}{d_1^4}\ge (1+o(1))u\]
    as $p\to\infty$, since $d_1^2\le d_1^2d_2\le p+1+2\sqrt{p}$ by the Hasse--Weil bound. Hence
   \begin{eqnarray*}
     s(\Ell(p))&\gg&\sum_{u\mid d_1\mid p-1} \frac{\varphi(u)\tau(d_1/u)}{d_1^3}\sum_{k\mid d_1^2/u}\left(\log \left((1+o(1))u\right)+2\gamma\right)\varphi(k)\\
               &\gg&\sum_{u\mid d_1\mid p-1} \frac{\varphi(u)\tau(d_1/u)}{d_1u}.
   \end{eqnarray*}
\end{proof}
\section{The number of cyclic subgroups ($h=c$)}\label{sec:hc}
\subsection{Number of cyclic subgroups of an abelian group of rank at most 2}
Similarly to Proposition \ref{prop:scRank2}, we have:
\begin{proposition}\label{prop:cRank2} For all integers $d_1,d_2\ge 1$,
    \begin{eqnarray*}
    c(\Z/d_1\times\Z/d_1d_2)&=&\sum_{u\mid d_1} (\varphi*\mu)(u)\tau(d_1/u)\tau(d_1^2d_2/u).
  \end{eqnarray*}
\end{proposition}
\begin{proof}
  By \cite[Corollary 1]{Toth12}, the number of cyclic subgroups of $\Z/m\times\Z/n$ is
    \begin{eqnarray*}
      \sum_{\substack{a\mid m \\ b\mid n}} \varphi((a,b))&=&\sum_{\substack{a\mid m \\ b\mid n}}\sum_{d\mid (a,b)} \mu(d) \frac{(a,b)}{d}=\sum_{\substack{a\mid m \\ b\mid n}}\sum_{e,d\mid (a,b)} \frac{\mu(d)}{d}\varphi(e)\\
                                                         &=&\sum_{u\mid (a,b)}(\varphi*\mu)(u)\tau(m/u)\tau(n/u).
    \end{eqnarray*}
    Alternatively, see \cite[Theorem 5]{HampHolTothWies14}.
\end{proof}

\subsection{The densities $w_{s,p}$ in arithmetic progressions}

We note that the convolution $\varphi*\mu$ is multiplicative and
\[(\varphi*\mu)(\ell^e)=
  \begin{cases}
    \ell-2&:e=1\\
    \ell^{e-2}(\ell-1)^2&:e\ge 2.
  \end{cases}\]

Similarly to the case $h=s$, we find that:
\begin{proposition}
  Proposition \ref{prop:Chs} holds with $h=s$ replaced by $c$, up to changing $\varphi(u)$ to $(\varphi*\mu)(u)$ in $C_{h,p}^{(1)}$.
\end{proposition}

Since $0\le (\varphi*\mu)(u)\le u$, the bounds obtained in Section \ref{sec:hs} still apply, up to Lemma \ref{lemma:EsG}. For that one, it suffices to note that
\begin{eqnarray*}
  &&\sum_{d_1\mid p-1} \frac{1}{d_1^2}\sum_{u\mid d_1}(\varphi*\mu)(u)\tau(d_1/u)\tau(d_1^2/u)\\
  &\ll&\sum_{d_1\mid p-1} \frac{\tau(d_1^2)}{d_1^2}(\sigma*\id)(d_1)\ll\log_2{p}\sum_{d_1\mid p-1} \frac{\tau(d_1^2)}{d_1}.
\end{eqnarray*}
All in all, this gives the second part of Theorem \ref{thm:scEll}.
\subsection{Order of magnitude}
The proof of the second parts of Propositions \ref{prop:averagep} and \ref{prop:upperlowerBounds} is similar to the case $h=s$ (see Section \ref{sec:orderMags}).

\bibliographystyle{alpha}
\bibliography{references}

\begin{thebibliography}{HHTW14}

\bibitem[AG17]{AchGord17}
Jeffrey~D. Achter and Julia Gordon.
\newblock Elliptic curves, random matrices and orbital integrals.
\newblock {\em Pacific J. Math.}, 286(1):1--24, 2017.
\newblock With an appendix by S. Ali Altu\u{g}.

\bibitem[BHBS05]{BHBS05}
William~D. Banks, Roger Heath-Brown, and Igor~E. Shparlinski.
\newblock On the average value of divisor sums in arithmetic progressions.
\newblock {\em International Mathematics Research Notices}, (1):1--25, 2005.

\bibitem[Blo07]{Blo07}
Valentin Blomer.
\newblock The average value of divisor sums in arithmetic progressions.
\newblock {\em The Quarterly Journal of Mathematics}, 59(3):275--286, 2007.

\bibitem[BM97]{BhowMenz97}
Gautami Bhowmik and Hartmut Menzer.
\newblock On the number of subgroups of finite {Abelian} groups.
\newblock In {\em Abhandlungen aus dem {Mathematischen} {Seminar} der
  {Universität} {Hamburg}}, volume~67, pages 117--121. Springer, 1997.

\bibitem[BPS96]{BPS96}
Alexandre~V. Borovik, Laszlo Pyber, and Aner Shalev.
\newblock Maximal subgroups in finite and profinite groups.
\newblock {\em Transactions of the American Mathematical Society},
  348(9):3745--3761, 1996.

\bibitem[Cal87]{Calh87}
William~C. Calhoun.
\newblock Counting the subgroups of some finite groups.
\newblock {\em The American Mathematical Monthly}, 94(1):54--59, 1987.

\bibitem[Cox89]{Cox89}
David~A. Cox.
\newblock {\em Primes of the form {$x^2 + ny^2$}}.
\newblock John Wiley \& Sons, Inc., New York, 1989.

\bibitem[DKS17]{DKS17}
Chantal David, Dimitris Koukoulopoulos, and Ethan Smith.
\newblock Sums of {Euler} products and statistics of elliptic curves.
\newblock {\em Mathematische Annalen}, 368(1):685--752, 2017.

\bibitem[Gek03]{Gek03}
Ernst-Ulrich Gekeler.
\newblock Frobenius distributions of elliptic curves over finite prime fields.
\newblock {\em International Mathematics Research Notices}, 37:1999--2018,
  2003.

\bibitem[HHTW14]{HampHolTothWies14}
Mario Hampejs, Nicki Holighaus, László Tóth, and Christoph Wiesmeyr.
\newblock Representing and counting the subgroups of the group
  {$\Z/m\times\Z/n$}.
\newblock {\em Journal of Numbers}, 2014:1--6, 2014.

\bibitem[IK04]{IK04}
Henryk Iwaniec and Emmanuel Kowalski.
\newblock {\em Analytic number theory}.
\newblock Colloquium Publications. American Mathematical Society, 2004.

\bibitem[Ivi97]{Ivic97}
Aleksandar Ivić.
\newblock On the number of subgroups of finite abelian groups.
\newblock {\em Journal de Théorie des Nombres de Bordeaux}, 9(2):371--381,
  1997.

\bibitem[IZ14]{IvZh14}
Aleksandar Ivić and Wenguang Zhai.
\newblock On the {Dirichlet} divisor problem in short intervals.
\newblock {\em The Ramanujan Journal}, 33(3):447--465, 2014.

\bibitem[Jut84]{Jutila84}
Matti Jutila.
\newblock On the divisor problem for short intervals.
\newblock {\em Annales Universitatis Turkuensis. Series A. I.}, 186:23--30,
  1984.

\bibitem[Jut89]{Jutila89}
Matti Jutila.
\newblock Mean value estimates for exponential sums.
\newblock In Hans~Peter Schlickewei and Eduard Wirsing, editors, {\em Number
  {Theory}}, volume 1380, pages 120--136. Springer, 1989.

\bibitem[KP17]{KaPe17}
Nathan Kaplan and Ian Petrow.
\newblock Elliptic curves over a finite field and the trace formula.
\newblock {\em Proceedings of the London Mathematical Society. Third Series},
  115(6):1317--1372, 2017.

\bibitem[KS18]{KerrShpar18}
Bryce Kerr and Igor Shparlinski.
\newblock Bilinear sums of {K}loosterman sums, multiplicative congruences and
  average values of the divisor function over families of arithmetic
  progressions.
\newblock 2018.
\newblock Preprint, arXiv:1811.09329.

\bibitem[LW07]{LoWa07}
Jody~M. Lockhart and William~P. Wardlaw.
\newblock Determinants of matrices over the integers modulo $m$.
\newblock {\em Mathematics Magazine}, 80(3):207--214, 2007.

\bibitem[PV15]{PongVau15}
Prapanpong Pongsriiam and Robert~C. Vaughan.
\newblock The divisor function on residue classes {I}.
\newblock {\em Acta Arithmetica}, 168(4):369--381, 2015.

\bibitem[PV18]{PongVau17}
Prapanpong Pongsriiam and Robert~C. Vaughan.
\newblock The divisor function on residue classes {II}.
\newblock {\em Acta Arithmetica}, 182(2):133--181, 2018.

\bibitem[Sch87]{Schoo87}
Ren\'{e} Schoof.
\newblock Nonsingular plane cubic curves over finite fields.
\newblock {\em J. Combin. Theory Ser. A}, 46(2):183--211, 1987.

\bibitem[Sil09]{Silv09}
Joseph~H. Silverman.
\newblock {\em The arithmetic of elliptic curves}.
\newblock Number 106 in Graduate {Texts} in {Mathematics}. Springer, 2nd
  edition, 2009.

\bibitem[Ste92]{Steh92}
Thomas Stehling.
\newblock On computing the number of subgroups of a finite {Abelian} group.
\newblock {\em Combinatorica}, 12(4):475--479, 1992.

\bibitem[Tó12]{Toth12}
László Tóth.
\newblock On the number of cyclic subgroups of a finite {Abelian} group.
\newblock {\em Bulletin mathématique de la Société des Sciences
  Mathématiques de Roumanie}, 55 (103)(4):423--428, 2012.

\bibitem[Tă10]{Tar10}
Marius Tărnăuceanu.
\newblock An arithmetic method of counting the subgroups of a finite abelian
  group.
\newblock {\em Bulletin mathématique de la Société des Sciences
  Mathématiques de Roumanie}, pages 373--386, 2010.

\bibitem[Vl{\u{a}}99]{Vla99}
Serge~G. Vl{\u{a}}du{\c{t}}.
\newblock Cyclicity statistics for elliptic curves over finite fields.
\newblock {\em Finite Fields and Their Applications}, 5(1):13 -- 25, 1999.

\end{thebibliography}
\end{document}